\documentclass[12pt,authoryear]{article}
\usepackage[utf8]{inputenc}

\usepackage{graphicx}
\usepackage{tikz}
\usetikzlibrary{decorations.pathreplacing}
\usetikzlibrary{shapes.geometric}
\usepackage{enumitem}
\usepackage{lscape}
\usepackage[round, sort&compress]{natbib}
\usepackage[margin=3cm]{geometry}

\usepackage{xcolor}
\usepackage{xspace}

\usepackage{amsmath}
\usepackage{amssymb}
\usepackage{amsthm}
\usepackage[mathscr]{euscript}
\usepackage{bbm}
\usepackage{dsfont}

\newtheorem{lemma}{Lemma}
\newtheorem{theorem}{Theorem}
\newtheorem{prop}{Proposition}

\theoremstyle{definition}
\newtheorem{rmk}{Remark}

\newcommand{\Prob}{\mathbb{P}}
\newcommand{\E}{\mathbb{E}}

\newcommand{\midd}{\,\middle|\,} % big conditioning bar
\newcommand{\I}[1]{\mathbbm{1}_{\{#1\}}}
\newcommand{\1}[1]{\mathbbm{1}_{#1}}
\newcommand{\Exp}{\operatorname{Exp}}
\newcommand{\ON}{1_N}

\usepackage{hyperref}
\begin{document}


\title{Weak Convergence of Non-neutral Genealogies\\ to Kingman's Coalescent}

\author{Suzie Brown$^{1}$\\s.brown.18@warwick.ac.uk \and Paul A. Jenkins$^{1,2,3}$\\p.jenkins@warwick.ac.uk \and Adam M. Johansen$^{1,3}$\\a.m.johansen@warwick.ac.uk \and Jere Koskela$^{1}$\\j.koskela@warwick.ac.uk} 
%\date{Department of Statistics, University of Warwick, Coventry, U.K.}
{\date{\small $^1$ Department of Statistics, University of Warwick, Coventry, U.K.\\[2pt]
$^2$ Department of Computer Science, University of Warwick, Coventry, U.K.\\[2pt]
$^3$ Alan Turing Institute, London, U.K.}}
      
\maketitle

\begin{abstract}

Interacting particle populations undergoing repeated mutation and fitness-based selection steps model genetic evolution, and describe a broad class of sequential Monte Carlo methods. 
The genealogical tree embedded into the system is important in both applications. 
Under neutrality, when fitnesses of particles and their parents are independent, rescaled genealogies are known to converge to Kingman's coalescent. 
Recent work established convergence under non-neutrality, but only for finite-dimensional distributions. 
We prove weak converge of non-neutral genealogies on the space of c\`adl\`ag paths under standard assumptions, enabling analysis of the whole genealogical tree. 
The proof relies on a conditional coupling in a random environment.

\end{abstract}


% \begin{highlights}
% \end{highlights}

%\begin{keyword}
%coalescent \sep interacting particle system \sep limit theorem \sep selection \sep sequential Monte Carlo

%\MSC[2020] 60J90 \sep 65C35 \sep 92D15

%\end{keyword}

\section{Introduction}\label{sec:intro}

The \emph{$n$-coalescent} \citep{kingman1982coal, kingman1982exch,kingman1982gene} is a homogeneous continuous-time Markov process on the space $\mathcal{P}_n$ of partitions of $\{1,\dots,n\} =: [ n ]$.
The non-zero entries of its infinitesimal generator $Q$ are
\begin{equation*}
q_{\xi,\eta} = \begin{cases}
-|\xi|(|\xi|-1)/2 & \xi=\eta \\
1 & \xi \prec \eta
\end{cases}
\end{equation*}
for every $\xi, \eta \in \mathcal{P}_n$, where $|\xi|$ denotes the number of blocks in $\xi$, and $\xi \prec \eta$ means that $\eta$ is obtained from $\xi$ by merging exactly two blocks.
It is the limiting genealogical process, as the population size goes to infinity, for samples of $n$ individuals from a wide range of population models.
The original work of \citet{kingman1982gene} provides sufficient conditions for convergence of genealogies from \emph{Cannings models} to the $n$-coalescent, in the sense of finite-dimensional distributions.
Cannings models are characterised by a fixed population size, exchangeable offspring counts, and i.i.d.\ generations.

\citet{mohle1998} provides sufficient conditions for the wider class of models in which the population size may vary deterministically, the offspring distributions are independent (but not i.i.d.) across generations, and exchangeability is replaced by the weaker \emph{random assignment condition}.
Independence of family sizes in different generations is incompatible with hereditary fitness, and essentially implies \emph{neutral} reproduction \citep{delmoral2009}.
For that class of models and under the same conditions, \citet{mohle1999} proves weak convergence of the genealogies as stochastic processes.
\citet{mohle2000} gives a simpler condition which is necessary and sufficient for weak convergence of genealogies from Cannings models to the Kingman coalescent.

We consider a still wider class of models where exchangeability is relaxed to random assignment, and independence between generations is not required, so that our results apply to non-neutral models.
This class was also treated in \citet{brown2021}, where convergence of finite-dimensional distributions was proved under a non-neutral analogue of the condition of \citet{mohle2000}.
Here we prove weak convergence under the same condition.
Our proof follows the structure of \citet[Theorem 3.1]{mohle1999}, but removing the assumption of independent family sizes between generations results in considerable technical complications because the pre-limiting, reverse-time genealogies are no longer Markov processes.
Non-Markovianity rules out some standard weak convergence conditions, such as uniform convergence of semigroups \cite[Chapter 4, Theorem 2.5]{ethier1986}.
We overcome these complications and prove weak convergence of the non-Markovian genealogical processes to the Markovian $n$-coalescent limit by controlling the modulus of continuity of pre-limiting processes.
Our approach yields weak convergence without assuming Markovianity \cite[Chapter 3, Corollary 7.4 and Theorem 7.8]{ethier1986}.

The models studied are of interest not only in population genetics, but also in sequential Monte Carlo (SMC): a very broad class of algorithms used in computational statistics and related disciplines \citep[see e.g.][for an introduction]{chopin2020}.  
In this context, the model we study characterises the dynamics of a population of ``particles'' whose empirical measure approximates a sequence of measures of interest, such as the conditional distribution of the latent part of a hidden Markov model given some observations.
Genealogies induced by the resampling step of SMC are critical to the performance of these algorithms, as has been known since SMC was first introduced to the statistics literature \citep{gordon1993}.
Genealogical trees embedded into SMC particle systems have been the subject of numerous studies \citep{delmoral2001b, delmoral2009, delmoral2016}, see also \citet[Chapter 3]{delmoral2004}, but direct analysis of the marginal genealogies has only taken off recently \citep{brown2021, koskela2020annals}. 
Indeed, \citet[Section 4]{brown2021} verify that the conditions of our main result, Theorem~\ref{thm:weakconv} presented below, hold for several important classes of SMC algorithms.
Hence, Theorem~\ref{thm:weakconv} allows us to strengthen several corollaries of \citet{brown2021} to weak convergence of the genealogical processes, facilitating convergence statements for a larger class of test functions than those which depend only on the finite-dimensional distributions.
Prominent examples of functions which cannot be computed from finite-dimensional distributions alone include the time to the most recent common ancestor (TMRCA) and the total branch length of the genealogy, both of which measure the memory cost of storing algorithm output.
\citet{jacob2015} showed that the TMRCA and total branch length of all $N$ particles are both $O( N \log N )$, while \cite{koskela2020annals} showed that those for a sample of $n$ particles are $\Theta( N )$ and $\Theta( N \log n )$, respectively, strongly suggesting that the results of \citet{jacob2015} could be sharpened to $\Theta(N)$ and $\Theta(N \log N)$, respectively.
Because of the lack of a weak convergence result, the argument of \citet[Corollary 2]{koskela2020annals} relies on a cumbersome workaround involving a coupling of the genealogical process with a separate $n$-coalescent.
Our Theorem \ref{thm:weakconv} removes the need for such workarounds, and facilitates simpler a priori analysis of marginal SMC genealogies in other settings, such as design of variance estimation schemes \citep{olsson2019} and conditional SMC updates in particle MCMC \citep{andrieu2010}.
In both settings, designing a large enough particle system to guarantee that a large number of distinct ancestral lines remain after a given number of generations is crucial to practical performance.

\section{Encoding genealogies}

Consider an interacting particle system in which $N$ particles stochastically reproduce in discrete, non-overlapping generations, such that each particle in generation $t \in \mathbb{N}$ has a single parent in the previous generation.
For convenience, we label time in reverse throughout this article, with the terminal generation labelled $0$, their parents being in generation $1$, and so on.
The index of the generation-$t$ parent of individual $j$ in generation $t - 1$ is denoted $a_t^{(j)} \in [N]$, and the number of offspring that individual $i$ in generation $t$ has in generation $t-1$ is $\nu_t^{(i)} := |\{ j: a_t^{(j)} = i \}|$.
Let $(\mathcal{F}_t)_{t\in\mathbb{N}}$, where $\mathcal{F}_t = \sigma(\{\nu_{s}^{(1:N)}\}_{s=0}^t)$, be the reverse-time filtration generated by the vectors of offspring counts $\nu_t^{(1:N)} := ( \nu_t^{(1)}, \ldots, \nu_t^{(N)} )$.

We study the genealogies of finitely many individuals under the asymptotic regime in which $N\to\infty$. In particular, sample $n\leq N$ individuals from generation $0$ uniformly without replacement, and trace back the corresponding lineages to obtain their genealogy which, following \citet{kingman1982coal}, is encoded by a $\mathcal{P}_n$-valued stochastic process $(G_t^{(n,N)})_{t\in\mathbb{N}}$ with initial value $G_0^{(n,N)} = \{ \{1\}, \dots, \{n\} \} =: \Delta $.
At each $t\in\mathbb{N}$, two indices $i\neq j \in [n]$ are in the same block of $G_t^{(n,N)}$ if and only if terminal particles $i$ and $j$ share a common ancestor at time $t$ (i.e.\ $t$ generations back).

Under the assumption \ref{standing_assumption} stated below, it is sufficient for our purposes to consider only offspring counts $\nu_t^{(1:N)}$ rather than the parental indices $a_t^{(1:N)}$, the latter being generally more informative.
\begin{enumerate}[label=(A\arabic*)]
\item\label{standing_assumption} The conditional distribution of parental indices $a_t^{(1:N)}$ given offspring counts $\nu_t^{(1:N)}$ is uniform over all assignments such that $ |\{ j: a_t^{(j)} =i \}|= \nu_t^{(i)} $ for all $i$.
\end{enumerate}
Known as the \emph{random assignment condition},
\ref{standing_assumption} is weaker than exchangeability of the particles within a generation.
As we will see, it is still sufficient to yield an exchangeable coalescent process in the $N \to \infty$ limit.

In order to obtain a well-defined limit for the genealogical process as $N\to\infty$, we must scale time by a suitable function $\tau_N(\cdot)$.
To define this time scale, we first define the conditional pair merger probability,
\begin{equation}\label{eq:defn_cN}
c_N(t) := \frac{1}{(N)_2} \sum_{i=1}^N (\nu_t^{(i)})_2,
\end{equation}
where $(n)_k := n (n - 1) \cdots (n - k + 1)$ is the falling factorial.
This is the probability, conditional on $\nu_t^{(1:N)}$, that a randomly chosen pair of lineages in generation $t$ merges exactly one generation earlier.
The interpretation of $c_N(t)$ as a conditional merger probability is justified by assumption \ref{standing_assumption}, and the same is true of the interpretation of $D_N(t)$ in \eqref{eq:defn_DN} as an upper bound on the probability of larger mergers.
To achieve a limiting pair merger rate of 1, as in the $n$-coalescent, we rescale time by the left-inverse:
\begin{equation}\label{eq:defn_tauN}
\tau_N(t) := \inf \left\{ s \in \mathbb{N} : \sum_{r=1}^s c_N(r) \geq t \right\} .
\end{equation}
The function $\tau_N$ maps continuous to discrete time, providing the link between the discrete-time generations and the continuous-time scaling limit.

Our definition of $c_N(t)$ differs from that which is usual in the population genetics literature, where the deterministic function defined as the expectation of \eqref{eq:defn_cN} is identified as $c_N(t)$.
Our definition yields a time scale which is random, depending on the realisation of $\nu_t^{(1:N)}$ for each $t$.
In the context of SMC, this is necessary to accommodate the heterogeneity of the system: the random time scale subsumes the time-inhomogeneity of the model, so that the rescaled ancestry converges to a time homogeneous process.
Inhomogeneous time scales have been studied in the genetics context as models of variable population size, but even in these settings $c_N(t)$ is defined as an expectation so that the inhomogeneous time scale varies deterministically \citep{mohle2002coal}.

We will also make extensive use of the following upper bound \citep{koskela2020annals} on the conditional probability of a multiple merger (three or more lineages merging, or two or more simultaneous mergers):
\begin{equation}\label{eq:defn_DN}
D_N(t) := \frac{1}{N(N)_2} \sum_{i=1}^N (\nu_t^{(i)})_2
        \left\{ \nu_t^{(i)} + \frac{1}{N} \sum_{j\neq i} (\nu_t^{(j)})^2 \right\} .
\end{equation} 
This is used to control the rate of multiple mergers, which must be dominated by the pair-merger rate as $N\to\infty$ if we are to recover an $n$-coalescent in the limit.
Proposition~\ref{thm:cN_properties} collects some basic properties of $c_N$, $D_N$ and $\tau_N$.
\begin{prop}\label{thm:cN_properties}
For all $t\in\mathbb{N}$, $t^\prime > s^\prime \geq 0$,
\begin{enumerate}[label=$(\alph*)$]
\item \label{item:cN_property1} \hspace{5pt}
    $\begin{aligned}
    0 \leq D_N(t) \leq c_N(t) \leq 1,
    \end{aligned}$
\item \label{item:cN_property4} \hspace{5pt}
    $\begin{aligned}
    t^\prime - (s^\prime + 1) \mathds{1}_{ ( 0, \infty ) }( s^\prime )
    \leq \sum_{r=\tau_N(s^\prime) + 1}^{\tau_N(t^\prime)} c_N(r)
    \leq t^\prime +1,
    \end{aligned}$
\item \label{item:cN_property6} \hspace{5pt}
    $\begin{aligned}
    \tau_N(t^\prime) \geq t^\prime.
    \end{aligned}$
\end{enumerate}
\end{prop}

\begin{proof}
The outermost bounds in \ref{item:cN_property1} follow from the fact that the entries of $\nu_t^{(1:N)}$ are non-negative and sum to $N$.
The central inequality follows as outlined in \citet[Supplement, p.\ 13--14]{koskela2022erratum}.
The case $s^{\prime} = 0$ in \textbf{\ref{item:cN_property4}} follows directly from the definition of $\tau_N$ in \eqref{eq:defn_tauN} and part \ref{item:cN_property1}, while the case $s^{\prime} > 0$ is obtained by applying the $s^{\prime} = 0$ case to both sums in
\begin{equation*}
\sum_{r=\tau_N(s^\prime)+1}^{\tau_N(t^\prime)} c_N(r)
= \sum_{r=1}^{\tau_N(t^\prime)} c_N(r) 
        - \sum_{r=1}^{\tau_N(s^\prime)} c_N(r) .
\end{equation*}
Finally, \ref{item:cN_property6} follows from \ref{item:cN_property1} and the definition of $\tau_N$ in \eqref{eq:defn_tauN}.
\end{proof}

We recall the following lemma, proved in \citet[Lemma 3.2]{brown2021thesis}.
\begin{lemma}\label{thm:kjjslemma2}
Fix $t>0$, and recall that $(\mathcal{F}_r)$ is the backwards-in-time filtration generated by the offspring counts $\nu_r^{(1:N)}$ at each generation $r$.
Let 
\begin{equation*}
    f_N : \Bigg\{ \nu^{(1:N)} : \nu^{(i)} \geq 0 \text{ for each } i \in [N], \text{ and } \sum_{ i = 1 }^N \nu^{(i)} = N \Bigg\} \mapsto \mathbb{R}.
\end{equation*}
Then 
\begin{equation*}
\E \Bigg[ \sum_{r=1}^{\tau_N(t)} f_N(\nu_r^{(1:N)}) \Bigg] 
= \E \Bigg[ \sum_{r=1}^{\tau_N(t)} \E [ f_N(\nu_r^{(1:N)}) \mid \mathcal{F}_{r-1} ] \Bigg] .
\end{equation*}
\end{lemma}

Let $p_{\xi\eta}(t)$ denote the conditional transition probabilities of the genealogical process from $\xi\in\mathcal{P}_n$ to $\eta\in\mathcal{P}_n$ given $\nu_t^{(1:N)}$, for each $t\in\mathbb{N}$.
The only non-zero transition probabilities $p_{\xi\eta}(t)$ are those where $\eta$ can be obtained from $\xi$ by merging some blocks of $\xi$ (i.e.\ some lineages coalescing).
Ordering the blocks by their least element, denote by $b_i$ the number of blocks of $\xi$ that merge to form block $i$ in $\eta$, for each $i \in [|\eta|]$. Hence $b_1 + \cdots + b_{|\eta|} = |\xi|$.
Then, assuming \ref{standing_assumption}, the conditional transition probability is given by
\begin{equation}\label{eq:defn_pxieta}
p_{\xi\eta}(t) 
:= \frac{1}{(N)_{|\xi|}} \sum_{\substack{i_1 , \ldots , i_{|\eta|} =1 
        \\ \text{all distinct} }}^N
        (\nu_t^{(i_1)})_{b_1} \cdots (\nu_t^{(i_{|\eta|})})_{b_{|\eta|}} .
\end{equation}
We will only need to work directly with the identity transition probabilities $p_{\xi\xi}(t)$.
Following \citet[Lemma 3.6]{brown2021}, but keeping the terms in $N$ explicit, yields the following lower bound.
\begin{prop}
\label{thm:pDelta_LB}
Let $\xi \in \mathcal{P}_n$, $N>2$. Then
\begin{equation*}
p_{\xi\xi}(t)
\geq 1 - \binom{|\xi|}{2} \frac{N^{|\xi|-2}}{(N-2)_{|\xi|-2}} \left[ c_N(t) + B_{|\xi|} D_N(t) \right],
\end{equation*}
where 
\begin{equation*}
B_{|\xi|} = K (|\xi|-1)! (|\xi|-2) \exp\big( 2 \sqrt{2(|\xi|-2)} \big)
\end{equation*}
for some $K>0$ that does not depend on $|\xi|$ or $N$.
\end{prop}

Define the asymptotic notation $\ON := (1 + O(N^{-1}))$, that is,
$f(N) = \ON$ if there exists $M>0$ such that for sufficiently large $N$, $|f(N) -1| \leq M/N$.
Throughout the remainder of the paper, we assume $N$ is large enough that any $\ON$ terms are positive.
The following upper bound can be found in \citet[Supplement, Lemma 1 Case 1]{koskela2022erratum}.
\begin{prop}
\label{thm:pDelta_UB}
Let $\xi \in \mathcal{P}_n$ and $B_{|\xi|}^\prime := \binom{|\xi|-1}{2}$. Then, for large enough $N$,
\begin{equation*}
p_{\xi\xi}(t)
\leq 1 - \ON \binom{|\xi|}{2}
        \left[ c_N(t) - B_{|\xi|}^\prime D_N(t) \right].
\end{equation*}
\end{prop}

\section{The convergence theorem}

In this section we give our main result, starting by defining a suitable metric space.
Denote by $\mathcal{D}$ the set of functions mapping $[0,\infty)$ to $\mathcal{P}_n$ that are right-continuous with left limits.
Our rescaled genealogical process $(G^{(n,N)}_{\tau_N(t)})_{t\geq0}$ and the $n$-coalescent are piecewise-constant functions mapping time $t\in[0,\infty)$ to a partition, and thus live in $\mathcal{D}$.
Finally, equip $\mathcal{P}_n$ with the discrete metric: for any $\xi, \eta \in \mathcal{P}_n$,
\begin{equation*}
\rho(\xi,\eta) 
= 1- \delta_{\xi\eta} 
:= \begin{cases}
    0 &\text{if } \xi=\eta, \\
    1 &\text{otherwise},
\end{cases}
\end{equation*}
and $\mathcal{D}$ with the associated Skorokhod (J1) topology and its Borel $\sigma$-algebra.

\begin{theorem}\label{thm:weakconv}
Let $\nu_t^{(1:N)}$ denote the offspring numbers in an interacting particle system satisfying \textup{\ref{standing_assumption}} and such that, for any sufficiently large $N$ and for all $t \in [ 0, \infty )$, $\Prob[ \tau_N(t) = \infty ] =0$. Suppose that there exists a deterministic sequence $(b_N)_{N\in\mathbb{N}}$ such that ${\lim}_{N\to\infty} b_N =0$ and, for all large enough $N$,
\begin{equation}\label{eq:mainthmcondition2}
\frac{1}{(N)_3} \sum_{i = 1}^N \E\left[ (\nu_t^{(i)})_3 \midd \mathcal{F}_{t - 1} \right]  \leq b_N \frac{1}{(N)_2} \sum_{i = 1}^N \E\left[ (\nu_t^{(i)})_2 \midd \mathcal{F}_{t - 1} \right],
\end{equation}
almost surely, uniformly in $t \geq 1$.
Then the rescaled genealogical process $(G_{\tau_N(t)}^{(n,N)})_{t\geq0}$ converges weakly in $\mathcal{D}$ to Kingman's $n$-coalescent as $N \to \infty$.
\end{theorem}

\begin{rmk}
Condition \eqref{eq:mainthmcondition2} is very natural: it requires the rate of (combinations of) mergers involving three or more lineages to be vanishingly small in comparison to that of binary mergers.
The sequence $( b_N )_{ N \in \mathbb{N} }$ controls the rate at which the ratio of rates of large and binary mergers vanishes, and can decay to zero arbitrarily slowly.
The exact values of its entries are not special; it is just the sequence that plays an implicit role in the usual little-o notation.
As mentioned in Section \ref{sec:intro}, a natural analogue of \eqref{eq:mainthmcondition2} is known to be necessary and sufficient for convergence to the $n$-coalescent in the neutral case \citep[Equation (16)]{mohle2000}.
\end{rmk}

\begin{proof}[Proof of Theorem~\ref{thm:weakconv}]
The structure of the proof follows \citet{mohle1999}, albeit with considerable technical complication due to the dependence between generations (non-neutrality) in our model.
To make it digestible, the proof is broken down into a number of results which are organised into sections; the relationships between these are shown in Figure~\ref{fig:weakconv_structure}.

Convergence of the finite-dimensional distributions was proved in \citet[Theorem 3.2]{brown2021}.
Strengthening this to weak convergence on the space of processes amounts to establishing relative compactness of the sequence $\{ (G_{\tau_N(t)}^{(n,N)})_{t\geq0} \}_{N\in\mathbb{N}}$.
Since $\mathcal{P}_n$ is finite and therefore complete and separable, and the sample paths of $(G_{\tau_N(t)}^{(n,N)})_{t\geq0}$ live in $\mathcal{D}$, we can apply \citet[Chapter 3, Corollary 7.4]{ethier1986} which states that a sequence of processes $\{ (G_{\tau_N(t)}^{(n,N)})_{t\geq0} \}_{N\in\mathbb{N}}$ is relatively compact if and only if the following two conditions hold:
\begin{enumerate}
\item \label{item:relcomp1} For every $\epsilon>0$ and rational $t\geq 0$, there exists a compact set $\Gamma_{ \epsilon, t } \subseteq \mathcal{P}_n$ such that
\begin{equation*}
\liminf_{N\to\infty} \Prob\left[ G_{\tau_N(t)}^{(n,N)} \in \Gamma_{ \epsilon, t } \right] 
\geq 1-\epsilon.
\end{equation*}
\item \label{item:relcomp2} For every $\epsilon>0$, $t>0$ there exists $\delta>0$ such that
\begin{equation}\label{modulus_of_continuity}
\liminf_{N\to\infty} \Prob\left[ \omega\left( G_{\tau_N(\cdot)}^{(n,N)}, \delta, t \right) < \epsilon \right] 
\geq 1-\epsilon,
\end{equation}
where $\omega$ is the modified modulus of continuity:
\begin{equation*}
\omega\left( G_{\tau_N(\cdot)}^{(n,N)}, \delta, t \right) := \inf \max_{i \in [K]} 
        \sup_{u,v \in [T_{i-1}, T_i)} \rho\left( 
        G_{\tau_N(u)}^{(n,N)}, G_{\tau_N(v)}^{(n,N)} \right),
\end{equation*}
with the infimum taken over all partitions of the form $0=T_0<T_1<\cdots <T_{K-1} <t \leq T_K$ (for any $K$) such that $\min_{i\in[K]} (T_i - T_{i-1}) > \delta$. 
\end{enumerate}
Since $\mathcal{P}_n$ is finite and hence compact, Condition~\ref{item:relcomp1} is satisfied automatically with $\Gamma_{ \epsilon, t } = \mathcal{P}_n$. 
Intuitively, Condition~\ref{item:relcomp2} ensures that the jumps of the process are well-separated. 
In our case, $\rho( G_{\tau_N(u)}^{(n,N)}, G_{\tau_N(v)}^{(n,N)} ) = 1$ if there is at least one jump between times $u$ and $v$, and 0 otherwise.
The supremum and maximum indicate whether there is a jump inside any of the intervals of the given partition; this can be zero only if all of the jumps up to time $t$ occur exactly at the times $T_{ 0 : K }$. 
The infimum over all allowed partitions can only equal zero if no two jumps occur less than $\delta$ (unscaled) time apart.

To prove Theorem \ref{thm:weakconv}, it remains to verify Condition \ref{item:relcomp2}.
To do this, we use a coupling with a process for which Condition \ref{item:relcomp2} is easier to check, and which will imply that it also holds for the genealogical process of interest.
Define $p_t := \max_{\xi\in \mathcal{P}_n} \{1 - p_{\xi\xi}(t)\} = 1 - p_{\Delta\Delta}(t)$, where $\Delta = \{ \{1\},\dots, \{n\} \}$.
For a proof that the maximum is attained at $\xi = \Delta$, see Lemma~\ref{thm:maximum_pr}. 
Following \citet{mohle1999}, we construct the process $(Z_t, S_t)_{t \in \mathbb{N}_0}$ on $\mathbb{N}_0 \times \mathcal{P}_n$ with initial state $( Z_0, S_0 ) = ( 0, \Delta )$, and conditional transition probabilities
\begin{align}
\Prob[ Z_t = j , S_t = \eta \mid& Z_{t-1} = i, S_{t-1} = \xi, \mathcal{F}_\infty ] \notag\\
&= \begin{cases}
1 - p_t &\quad \text{if } j=i \text{ and } \eta=\xi, \\
p_{\xi\xi}(t) + p_t - 1  &\quad \text{if } j=i+1 \text{ and } \eta=\xi, \\
p_{\xi\eta}(t) &\quad \text{if } j=i+1 \text{ and } \eta\neq\xi, \\
0 &\quad \text{otherwise}.
\end{cases}
\label{eq:constructed_chain}
\end{align}
The definition of $p_t$ ensures that the second case of \eqref{eq:constructed_chain} is non-negative, attaining the value zero when $\xi=\Delta$.

Unlike the corresponding process in \citet{mohle1999}, the transition probabilities in \eqref{eq:constructed_chain} depend on offspring counts. Thus, $(Z_t, S_t)$ is only Markovian conditional on $\mathcal{F}_\infty$, and can be thought of as a time-inhomogeneous Markov chain in a random environment.
Marginally, $(S_t)$ has the same distribution as the genealogical process of interest, while $(Z_t)$ jumps whenever $(S_t)$ does, but also has some extra jumps.
The jump times of $(Z_t)$ do not depend on the current state, making it much easier to analyse.
Our construction also resembles that of \citet{mohle2002coal}, where the process $(S_t)$ is also time-inhomogeneous but still Markovian without conditioning on a random environment.

The coupling of jumps implies that the modulus of continuity of $(Z_t)$ is at least as large as that of $(S_t)$.
Hence, we will show that \eqref{modulus_of_continuity} holds for $(Z_t)$, and conclude that it also holds for the genealogical processes of interest.
Denote by $0=T_0^{(N)}<T_1^{(N)}<\dots$ the jump times of the rescaled process $(Z_{\tau_N(t)})_{t\geq0}$, and by $\varpi_i^{(N)} := T_i^{(N)} - T_{i-1}^{(N)}$ the corresponding holding times.
Suppose that for some fixed $\varpi_1^{(N)}, \varpi_2^{(N)}, \dots$ and $t>0$, there exists $m \in \mathbb{N}$ and $\delta >0$ such that
$\varpi_i^{(N)} > \delta$ for all $i\in [m]$, and
$T_m^{(N)} \geq t$.
Then
$K_N := \min\{i : T_i^{(N)} \geq t\}$
is well-defined with $1\leq K_N \leq m$,
and $T_{ 1 : K_N }^{(N)}$ is a partition of the form required for Condition~\ref{item:relcomp2}.
Indeed $(Z_{\tau_N(\cdot)})$ is constant on every interval $[ T_{i-1}^{(N)} , T_i^{(N)} )$ by construction, so $\omega( (Z_{\tau_N(\cdot))} , \delta, t ) = 0$. 
We therefore have that
for each $m \in \mathbb{N}$ and $\delta >0$,
\begin{equation*}
\Prob\left[ \omega\left( (Z_{\tau_N(\cdot)}) , \delta, t \right) < \epsilon \right]
\geq \Prob\left[ T_m^{(N)} \geq t, \varpi_i^{(N)} > \delta \,\forall i\in [m] \right] .
\end{equation*}
Thus, a sufficient condition for Condition~\ref{item:relcomp2} to hold is: 
for any $\epsilon>0$, $t>0$, there exist $m\in\mathbb{N}$, $\delta>0$ such that
\begin{equation}\label{eq:condition2b}
\liminf_{N\to\infty} \Prob\left[ T_m^{(N)} \geq t, \varpi_i^{(N)} > \delta \,\forall i\in [m]\right]
\geq 1-\epsilon .
\end{equation}
By Lemma~\ref{thm:holdingtimes_distn}, below, the limiting distributions of $\varpi_i^{(N)}$ are i.i.d.\ $\Exp(\alpha_n)$, where $\alpha_n := n(n-1)/2$, so
\begin{equation*}
\limsup_{N\to\infty} \Prob\left[ \varpi_i^{(N)} \leq \delta \right] 
= 1- \liminf_{N\to\infty} \Prob\left[ \varpi_i^{(N)} > \delta \right] 
= 1- e^{-\alpha_n \delta}
\end{equation*}
for each $i$, and
\begin{align*}
\limsup_{N\to\infty} \Prob\left[ T_m^{(N)} < t \right]
&= 1- \liminf_{N\to\infty} \Prob\left[ T_m^{(N)} \geq t \right] \\
&= 1- \liminf_{N\to\infty} \Prob\left[ \varpi_1^{(N)} + \cdots + \varpi_m^{(N)} \geq t \right] \\
&= 1- e^{-\alpha_n t} \sum_{i=0}^{m-1} \frac{(\alpha_n t)^i}{i!},
\end{align*}
using the series expansion for the Erlang CDF \citep[see for example][Chapter 15]{forbes2011}.
Now
\begin{align*}
\liminf_{N\to\infty} 
        \Prob[ T_m^{(N)} \geq t, &\varpi_i^{(N)} > \delta 
        \,\forall i\in [m] ] \\
&\geq 1 - \limsup_{N\to\infty} \Prob\left[ T_m^{(N)} < t \right] 
        - \sum_{i=1}^m \limsup_{N\to\infty} \Prob\left[ \varpi_i^{(N)} \leq \delta\right]\\
&= 1- \Bigg( 1- e^{-\alpha_n t} \sum_{i=0}^{m-1} 
        \frac{(\alpha_n t)^i}{i!} \Bigg) - m( 1- e^{-\alpha_n \delta} ) ,
\end{align*}
which can be made $\geq 1-\epsilon$ by taking $m$ sufficiently large and $\delta$ sufficiently small.
Since this argument applies for any $\epsilon$ and $t$, \eqref{eq:condition2b} is satisfied.
Hence, so is Condition~\ref{item:relcomp2}, and the proof is complete.
\end{proof}

The structure of the proof of weak convergence in Theorem \ref{thm:weakconv} resembles that of the neutral case \citep[Theorem 3.1]{mohle1999}.
The complications arising from non-neutrality have been subsumed into Lemma \ref{thm:holdingtimes_distn}, stated below.
As illustrated in Figure \ref{fig:weakconv_structure}, its proof relies on a number of technical results, which we state and prove in Section \ref{sec:technical_results}, and particularly in Section \ref{sec:main_components_of_induction}.

The neutral techniques of \cite{mohle1999} have also been used to establish weak convergence of genealogical processes to more general coalescent models featuring simultaneous mergers of more than two lineages \citep{mohle2001, mohle2003, sagitov2003}.
Proofs of these results invariably contain two main parts: convergence of finite-dimensional distributions (established via convergence of generators in the Markovian setting), and control of the modulus of continuity.
We expect that this strategy will allow similar results for convergence of non-neutral models to multiple merger coalescents: use an approach reminiscent of \citet[Theorem 1]{koskela2020annals} to establish convergence of finite-dimensional distributions and then adapt Theorem \ref{thm:weakconv} to control the relevant modulus of continuity.

Throughout the remainder of the manuscript, we write $x_{ 1 : k } \leq y_{ 1 : k }$ for vectors $x_{ 1 : k } := ( x_1, \ldots, x_k )$ and $y_{ 1 : k }$ if the inequality holds elementwise.

\begin{lemma}\label{thm:holdingtimes_distn}
Assume \eqref{eq:mainthmcondition2} holds.
Then, as $N\to\infty$, the finite-dimensional distributions of $\varpi_1^{(N)} , \varpi_2^{(N)} , \dots$ converge to those of $\varpi_1, \varpi_2, \dots$, where the $\varpi_i$ are independent $\Exp(\alpha_n)$-distributed random variables.
\end{lemma}

\begin{proof}
For any $k \in \mathbb{N}$, there is a continuous bijection between the jump times $T_{ 1 : k }^{(N)}$ and the holding times $\varpi_{ 1 : k }^{(N)}$.
Thus, convergence of $\varpi_{ 1 : k }^{(N)}$ to $\varpi_{ 1 : k }$ is equivalent to convergence of the jump times to $T_{ 1 : k }$, where $T_i := \varpi_1+\dots+\varpi_i$. 
We will work with the jump times, following \citet[Lemma 3.2]{mohle1999}.
The idea is to prove by induction that, for any $k\in\mathbb{N}$ and $t_{ 1 : k } > 0$,
\begin{equation}\label{eq:518}
\lim_{N\to\infty} \Prob\left[ T_{ 1 : k }^{(N)} \leq t_{ 1 : k } \right]
= \Prob[ T_{ 1 : k } \leq t_{ 1 : k } ] .
\end{equation}
Take the basis case $k=1$, for which $\Prob[ T_1 \leq t ] 
= \Prob[ \varpi_1 \leq t ] = 1 - e^{-\alpha_n t}$
and $T_1^{(N)} >t$ if and only if $( Z_{ \tau_N( t ) } )$ has no jumps up to time $t$:
\begin{equation*}
\Prob\left[T_1^{(N)} >t \right]
= \E\left[ \Prob[ T_1^{(N)} >t \mid \mathcal{F}_\infty ] \right]
= \E\Bigg[ \prod_{r=1}^{\tau_N(t)} (1-p_r) \Bigg] .
\end{equation*}
Lemma~\ref{thm:basis} shows that this probability converges to $e^{-\alpha_n t}$ as required.

For the induction step, assume that \eqref{eq:518} holds for some $k$. 
We have
\begin{align*}
\Prob\left[ T_{ 1 : ( k + 1 ) }^{(N)} \leq t_{ 1 : ( k + 1 ) } \right]
= \Prob\left[ T_{ 1 : k }^{(N)} \leq t_{ 1 : k } \right] - \Prob\left[ T_{ 1 :  k }^{(N)} \leq t_{ 1 : k }, T_{k+1}^{(N)} > t_{k+1} \right].
\end{align*}
The first term on the right-hand side converges to $\Prob[ T_{ 1 : k } \leq t_{ 1 : k } ]$ by the induction hypothesis, and it remains to show that
\begin{align}
\lim_{N\to\infty} 
        &\Prob\left[ T_{ 1 : k }^{(N)} \leq t_{ 1 : k }, T_{k+1}^{(N)} > t_{k+1} \right] = \Prob[ T_{ 1 : k }\leq t_{ 1 : k }, T_{k+1} > t_{k+1} ]. \label{eq:induction}
\end{align}
As shown in \citet[p.\ 459]{mohle1999},
\begin{equation*}
\Prob[ T_{ 1 : k }\leq t_{ 1 : k }, T_{k+1} > t_{k+1} ]
= \alpha_n^k e^{-\alpha_n t} 
        \sum_{\substack{i_1\leq \dots\leq i_{k-1}\\ \in \{0,\dots,k\} :
        \\ i_j \geq j \forall j}} 
        \prod_{j=1}^k \frac{(t_j - t_{j-1})^{i_j - i_{j-1}}}{(i_j - i_{j-1})! } ,
\end{equation*}
while the probability on the left-hand side of \eqref{eq:induction} can be written
\begin{align*}
\Prob\left[ T_{ 1 : k }^{(N)} \leq t_{ 1 : k }, T_{k+1}^{(N)} > t_{k+1} \right] &= \E\left[ \Prob[ T_{ 1 : k }^{(N)} \leq t_{ 1 : k }, T_{k+1}^{(N)} > t_{k+1} 
        \mid \mathcal{F}_\infty] \right] \\
&= \E \Bigg[ \sum_{\substack{r_1 <\dots< r_k :
        \\ r_i \leq \tau_N(t_i) \forall i}}
        \Big( \prod_{i=1}^k p_{r_i} \Big)
        \Bigg( \prod_{\substack{r=1 \\ \notin \{r_1,\dots,r_k\} }}^{\tau_N(t)} 
        (1-p_r) \Bigg) \Bigg] .
\end{align*}
That is, there are jumps at some times $r_{ 1 : k }$, and identity transitions at all other times.
A similar expression is derived in \citet{mohle1999}, but here we have an additional expectation because the probabilities $p_r$ depend on the random offspring counts.
Lemmata \ref{thm:inductionUB} and \ref{thm:inductionLB} show that this probability converges to the correct limit.
This completes the induction.
\end{proof}

\begin{landscape}
\begin{figure}[ht]
\centering 
\resizebox{\columnwidth}{!}{
\begin{tikzpicture}[>=stealth, thick]
% phantom line (to add space below figure)
\draw[white] (-15,-7)--(6,-7);

% nodes
\node[align=center,draw=gray] at (4,0) (c7r2) 
        {Theorem~\ref{thm:weakconv}\\[-3pt] \textcolor{gray}{\textsc{weak conv.}}};
\node[align=center] at (-6,-3) (c4r3) {Lemma~\ref{thm:inductionLB}\\[-3pt] 
        \textcolor{gray}{\textsc{induction}} \\[-5pt] \textcolor{gray}{\textsc{l.b.}} };
\node[align=center] at (-6,0) (c4r2) {Lemma~\ref{thm:inductionUB}\\[-3pt] 
        \textcolor{gray}{\textsc{induction}} \\[-5pt] \textcolor{gray}{\textsc{u.b.}} };
\node[align=center] at (-6,3) (c4r1) {Lemma~\ref{thm:basis}\\[-3pt] 
        \textcolor{gray}{\textsc{basis}} };
\node[align=center] at (4,3) (c7r1) {Brown et al. 2021\\[-3pt] Theorem 3.2\\[-3pt] \textcolor{gray}{\textsc{f.d.d. conv.}} };
\node at (-15,-2.25) (c1r1) {Lemma~\ref{thm:kjjslemma2}};
\node at (4,-3) (c7r3) {Lemma~\ref{thm:maximum_pr}};
\node at (0,0) (c6r1) {Lemma~\ref{thm:holdingtimes_distn}};
\node[align=center] at (-3,6) (c5r1) {Proposition~\ref{thm:pDelta_LB}\\[-3pt] 
        \textcolor{gray}{\textsc{trans. prob.}} \\[-5pt] \textcolor{gray}{\textsc{u.b.}} };
\node[align=center] at (-3,-6) (c5r2) {Proposition~\ref{thm:pDelta_UB}\\[-3pt] 
        \textcolor{gray}{\textsc{trans. prob.}} \\[-5pt] \textcolor{gray}{\textsc{l.b.}} };
\node[align=center] at (0,-6) (c6r2) {Lemma~\ref{thm:DCT_Fubini}\\[-3pt] 
        \textcolor{gray}{\textsc{d.c.t. cond.}} };
\node at (-9,-6) (c3r1) {Lemma~\ref{thm:induction_sumprodcN}};
\node at (-12,6) (c2r1) {Lemma~\ref{thm:sumprod1}\ref{thm:sumprod1_a}};
\node at (-12,4.5) (c2r2) {Lemma~\ref{thm:sumprod1}\ref{thm:sumprod1_b}};
\node at (-12,3) (c2r3) {Lemma~\ref{thm:sumprod2}};
\node at (-12,1.5) (c2r4) {Lemma~\ref{thm:sumprod3}};
\node at (-12,-1.5) (c2r6) {Lemma~\ref{thm:indicators_tau}};
\node at (-12,0) (c2r5) {Lemma~\ref{thm:indicators_cN}};
\node at (-12,-3) (c2r7) {Lemma~\ref{thm:lim_AandB}};
\node at (-12,-4.5) (c2r8) {Lemma~\ref{thm:indicators_DN}};
\node at (-12,-6) (c2r9) {Lemma~\ref{thm:indicators_c2}};

% brace
\draw[decorate, decoration={brace, amplitude=8}] (-10.9,0.2)--(-10.9, -3.2);

% arrows from column 1
\draw[->] (c1r1)--(c2r8);
\draw[->] (c1r1)--(c2r5);
% arrows from column 2
\draw[->] (c2r5)--(c2r6);
\draw[->] (c2r1)--(c4r2);
\draw[->] (c2r2)--(c4r1);
\draw[->, dotted] (c2r2)--(c3r1);
\draw[->] (c2r3)--(c4r1);
\draw[->, dotted] (c2r3)--(c4r2);
\draw[->] (c2r4)--(c4r1);
\draw[->, dotted] (c2r4)--(c4r3);
\draw[->] (c2r8)--(c4r3);
\draw[->] (c2r9)--(c3r1);
%\draw[->] (c2r7)--(c3r1);
% arrows from brace
\draw[->] (-10.6, -1.5)--(c4r1);
\draw[->] (-10.6, -1.5)--(c4r2);
\draw[->] (-10.6, -1.5)--(c4r3);
\draw[->] (-10.6, -1.5)--(c3r1);
% arrows from column 3
\draw[->] (c3r1)--(c4r3);
% arrows from column 4
\draw[->] (c4r1)--(c6r1);
\draw[->] (c4r2)--(c6r1);
\draw[->] (c4r3)--(c6r1);
% arrows from column 5
\draw[->] (c5r1)--(c4r1);
\draw[->] (c5r1)--(c4r2);
\draw[->] (c5r2)--(c4r1);
\draw[->] (c5r2)--(c4r2);
\draw[->] (c5r2)--(c4r3);
% arrows from column 6
\draw[->] (c6r1)--(c7r2);
\draw[->] (c6r2)--(c4r2);
\draw[->] (c6r2)--(c4r3);
% arrows from column 7
\draw[->] (c7r1)--(c7r2);
\draw[->] (c7r3)--(c7r2);
\end{tikzpicture}
}
\caption[Structure of weak convergence proof]{Graph showing dependencies between the lemmata used to prove weak convergence. Dotted arrows indicate dependence via a slight modification of the preceding lemma.}
\label{fig:weakconv_structure}
\end{figure}
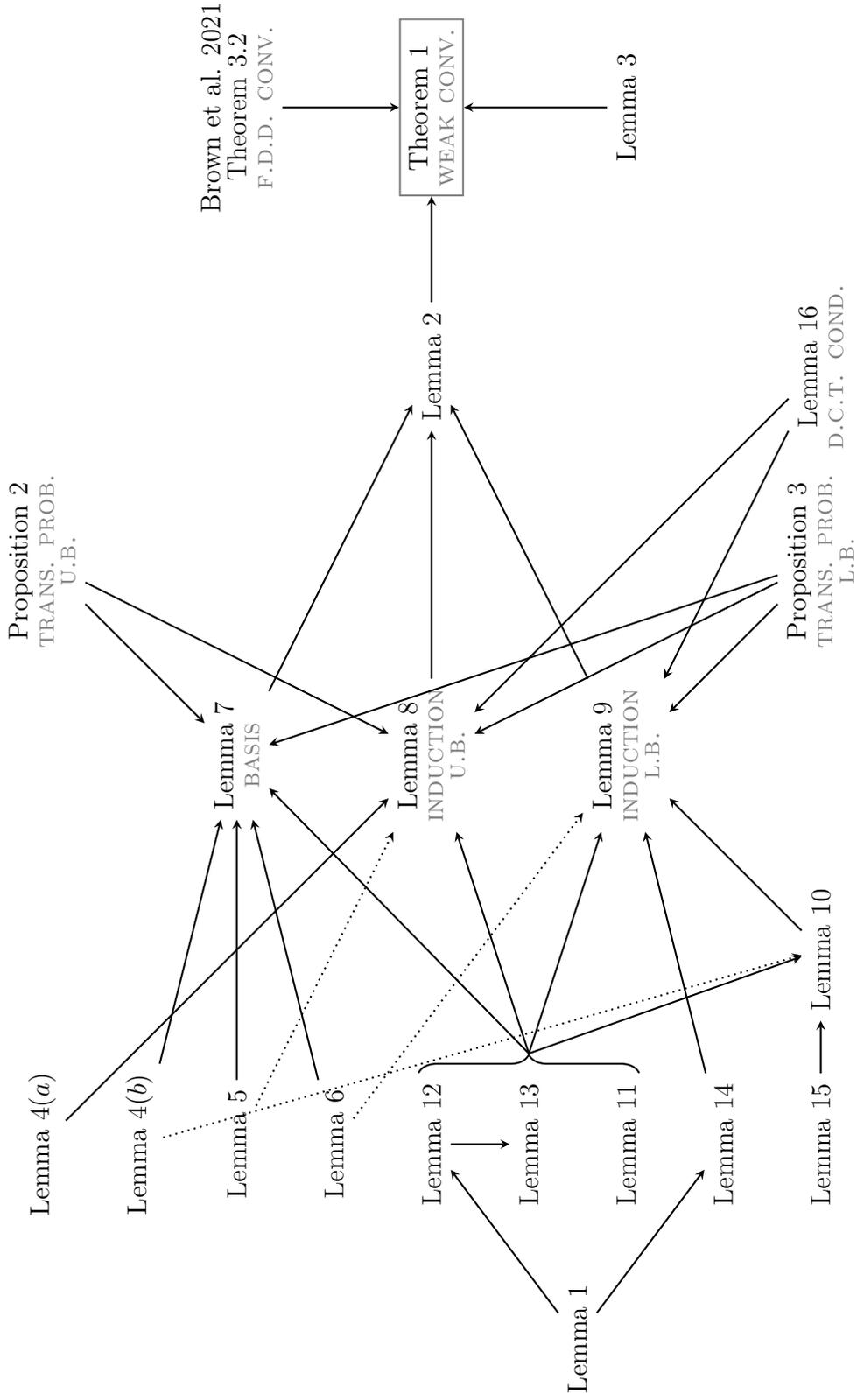
\end{landscape}

\section{Technical results}\label{sec:technical_results}

\begin{lemma}\label{thm:maximum_pr}
$\max_{\xi\in \mathcal{P}_n} \{1 - p_{\xi\xi}(t)\} = 1 - p_{\Delta\Delta}(t)$.
\end{lemma}

\begin{proof}
Consider any $\xi \in E$ consisting of $k$ blocks ($1\leq k\leq n-1$), and any $\xi^\prime\in E$ consisting of $k+1$ blocks. 
Setting $\eta=\xi$ in \eqref{eq:defn_pxieta},
\begin{equation*}
p_{\xi\xi}(t) 
= \frac{1}{(N)_k} \sum_{\substack{i_1,\dots,i_k=1 \\ \text{all distinct}}}^N 
        \nu_t^{(i_1)} \cdots \nu_t^{(i_k)} .
\end{equation*}
Similarly,
\begin{align*}
p_{\xi^\prime\xi^\prime}(t) &= \frac{1}{(N)_{k+1}} 
        \sum_{\substack{i_1,\dots, i_{k+1} =1 \\ \text{all distinct}}}^N 
        \nu_t^{(i_1)} \cdots \nu_t^{(i_k)} \nu_t^{(i_{k+1})} \\
&= \frac{1}{(N)_k(N-k)} \sum_{\substack{i_1,\dots,i_k =1 \\ \text{all distinct}}}^N 
        \Bigg\{ \nu_t^{(i_1)} \cdots \nu_t^{(i_k)} 
        \sum_{\substack{i_{k+1}=1 \\ \notin \{i_1,\ldots, i_k\} }}^N 
        \nu_t^{(i_{k+1})} \Bigg\} .
\end{align*}
Discarding the zero summands,
\begin{equation*}
p_{\xi^\prime\xi^\prime}(t) 
    = \frac{1}{(N)_k(N-k)} \sum_{\substack{i_1,\dots,i_k =1 \\ \text{all distinct:} 
        \\ \nu_t^{(i_1)},\dots,\nu_t^{(i_k)} > 0 }}^N
        \Bigg\{ \nu_t^{(i_1)} \cdots \nu_t^{(i_k)} 
        \sum_{\substack{i_{k+1}=1 \\ \notin \{i_1,\ldots, i_k\} }}^N 
        \nu_t^{(i_{k+1})} \Bigg\} .
\end{equation*}
The inner sum is
\begin{equation*}
\sum_{\substack{i_{k+1}=1 \\ \notin \{i_1,\ldots, i_k\} }}^N \nu_t^{(i_{k+1})} 
    = \Bigg\{ \sum_{i=1}^N \nu_t^{(i)} -  \sum_{i\in\{i_1,\dots,i_k\} } 
        \nu_t^{(i)} \Bigg\}
    \leq N - k ,
\end{equation*}
since $\nu_t^{(i_1)},\dots,\nu_t^{(i_k)} $ are all at least 1.
Hence
\begin{equation*}
p_{\xi^\prime\xi^\prime}(t)
    \leq  \frac{N-k}{(N)_k(N-k)} \sum_{\substack{i_1,\dots,i_k =1 
        \\ \text{all distinct:} \\ \nu_t^{(i_1)},\dots,\nu_t^{(i_k)} > 0 }}^N 
        \nu_t^{(i_1)} \cdots \nu_t^{(i_k)} 
    = p_{\xi\xi}(t) .
\end{equation*}
Thus, $p_{\xi\xi}(t)$ is decreasing in the number of blocks of $\xi$, and is therefore minimised by taking $\xi = \Delta$, which uniquely achieves the maximum $n$ blocks. This choice in turn maximises $1-p_{\xi\xi}(t)$, as required.
\end{proof}

\subsection{Bounds on sum-products}\label{sec:bounds_on_sum-products}

In this section we derive tractable bounds on sums of products of conditional merger probabilities, which themselves appear as upper and lower bounding envelopes of the conditional transition probability $p_r$ in Propositions \ref{thm:pDelta_LB} and \ref{thm:pDelta_UB}.
These sums of products can be regarded as building blocks of the conditional transition probabilities of the genealogical process, and the bounds obtained here facilitate proving its convergence. 
The sum-product bounds will be applied multiple times in the lemmata of this section.

\begin{lemma}\label{thm:sumprod1}
Fix $t > s > 0$ and $l\in\mathbb{N}$.
Then
\begin{enumerate}[label=$(\alph*)$]
\item \label{thm:sumprod1_a} %\hspace{0.2cm}
$\begin{aligned}[t]
\sum_{\substack{ s_1, \dots, s_l =\tau_N(s)+1 \\ \text{\normalfont{all distinct}} }}^{\tau_N(t)}
        \prod_{j=1}^l c_N(s_j)
    \leq (t - s + 1)^l \leq (t+1)^l,
\end{aligned}$
\item \label{thm:sumprod1_b} %\hspace{0.2cm}
$\begin{aligned}[t]
&\Bigg[ (t - s)^l - \Bigg( c_N(\tau_N(s)) + \binom{l}{2} \sum_{r = \tau_N(s) + 1}^{\tau_N(t)} c_N(r)^2 \Bigg) (t + 1)^l \Bigg]  \1{\{ c_N(\tau_N(s)) \leq t - s\}} \\
&\leq \sum_{\substack{ s_1, \dots, s_l =\tau_N(s) + 1 \\ \text{\normalfont{all distinct}} }}^{\tau_N(t)} \prod_{j=1}^l c_N(s_j) \leq (t - s)^l + c_N(\tau_N(t)) (t+1)^l.
\end{aligned}$
\end{enumerate}
\end{lemma}

\begin{proof} \textbf{\ref{thm:sumprod1_a}}
This follows from the inequalities
\begin{equation*}
\sum_{\substack{ s_1, \dots, s_l =\tau_N(s) + 1 \\ \text{all distinct} }}^{\tau_N(t)}
        \prod_{j=1}^l c_N(s_j)
\leq \Bigg( \sum_{r=\tau_N(s) + 1}^{\tau_N(t)} c_N(r) \Bigg)^l \leq (t - s + 1)^l,
\end{equation*}
the first of which follows from a multinomial expansion of the middle term and the second from Proposition~\ref{thm:cN_properties}\ref{item:cN_property4}. 

\textbf{\ref{thm:sumprod1_b}}
We begin by multiplying the bound in \citet[Supplement, equation (8)]{koskela2022erratum} by $\1{\{ c_N(s) \leq t - s\}}$, which is valid because the left-hand side is non-negative, and yields  
\begin{align*}
&\sum_{\substack{ s_1, \dots, s_l =\tau_N(s) + 1 \\ \text{all distinct} }}^{\tau_N(t)} 
        \prod_{j=1}^l c_N(s_j) \\
&\geq \Bigg[ \Bigg( \sum_{r=\tau_N(s) + 1}^{\tau_N(t)} c_N(r) \Bigg)^l \\
&\phantom{\geq \Bigg[} - \Bigg( \sum_{r=\tau_N(s) + 1}^{\tau_N(t)} c_N(r)^2 \Bigg)
        \binom{l}{2} \Bigg( \sum_{r=\tau_N(s) + 1}^{\tau_N(t)} c_N(r) \Bigg)^{l-2} \Bigg] \1{\{ c_N(\tau_N(s)) \leq t - s\}} \\
&\geq \Bigg[[t - s - c_N(\tau_N(s))]^l - \Bigg( \sum_{r=\tau_N(s) + 1}^{\tau_N(t)} c_N(r)^2 \Bigg)
        \binom{l}{2} (t + 1)^{l-2}\Bigg]\1{\{c_N(\tau_N(s)) \leq t - s\}},
\end{align*}
where the final inequality follows from the definition of $\tau_N$ and Lemma \ref{thm:sumprod1}\ref{thm:sumprod1_a}, and tracking the event $\{c_N(s) \leq t - s\}$ is necessary in case of a large negative value of the lower bound $t - s - c_N(\tau_N(s))$ when $l$ is even.
A binomial expansion of the first term on the right-hand side, followed by using Proposition \ref{thm:cN_properties}\ref{item:cN_property1}  results in
\begin{align*}
&\sum_{\substack{ s_1, \dots, s_l =\tau_N(s) + 1 \\ \text{all distinct} }}^{\tau_N(t)} 
        \prod_{j=1}^l c_N(s_j) \\
&\geq \Bigg[(t - s)^l - \sum_{i = 0}^{l - 1} \binom{l}{i} c_N(\tau_N(s))^{l - i} (t - s)^i \\
&\phantom{\geq \Bigg[ (t - s)^l} - \Bigg( \sum_{r=\tau_N(s) + 1}^{\tau_N(t)} c_N(r)^2 \Bigg) \binom{l}{2} (t + 1)^{l-2}\Bigg]\1{\{c_N(\tau_N(s)) \leq t - s\}} \\
&\geq \Bigg[ (t - s)^l - \Bigg( c_N(\tau_N(s)) + \binom{l}{2} \sum_{r = \tau_N(s) + 1}^{\tau_N(t)} c_N(r)^2 \Bigg) (t + 1)^l \Bigg]\1{\{c_N(\tau_N(s)) \leq t - s\}}.
\end{align*}

For the upper bound we have
\begin{align*}
\sum_{\substack{ s_1, \dots, s_l =\tau_N(s) + 1 \\ \text{all distinct} }}^{\tau_N(t)} 
        \prod_{j=1}^l c_N(s_j)
&\leq \Bigg( \sum_{r=\tau_N(s)+1}^{\tau_N(t)} c_N(s) \Bigg)^l \leq \left[ t - s + c_N(\tau_N(t)) \right]^l ,
\end{align*}
using the definition of $\tau_N$.
A binomial expansion and Proposition \ref{thm:cN_properties}\ref{item:cN_property1} yield
\begin{align*}
[ t - s + c_N(\tau_N(t)) ]^l &= (t - s)^l + \sum_{i=0}^{l-1} \binom{l}{i} (t - s)^i c_N(\tau_N(t))^{l-i}\\
&\leq (t - s)^l + c_N(\tau_N(t)) \sum_{i=0}^{l-1} \binom{l}{i} t^i \\
&\leq (t - s)^l + c_N(\tau_N(t)) (t+1)^l.
\end{align*}
\end{proof}
For later uses of Lemma \ref{thm:sumprod1} when $s = 0$, we emphasize that $c_N(\tau_N(0)) = 0$.

\begin{lemma}\label{thm:sumprod2}
Fix $t>0$, $l\in\mathbb{N}$.
Then, for any constant $B>0$,
\begin{align*}
\sum_{\substack{ s_1, \dots, s_l =1 \\ \text{\normalfont{all distinct}} }}^{\tau_N(t)} &
        \prod_{j=1}^l \left[ c_N(s_j) + B D_N(s_j) \right] \\
&\leq \sum_{\substack{ s_1, \dots, s_l =1 \\ \text{\normalfont{all distinct}} }}^{\tau_N(t)} 
        \prod_{j=1}^l c_N(s_j)
        + \Bigg( \sum_{s=1}^{\tau_N(t)} D_N(s) \Bigg) (t+1)^{l-1} (1+B)^l .
\end{align*}
\end{lemma}

\begin{proof}
We start with the binomial expansion
\begin{align}
&\sum_{\substack{ s_1, \dots, s_l =1 \\ \text{all distinct} }}^{\tau_N(t)} 
        \prod_{j=1}^l \left[ c_N(s_j) + B D_N(s_j) \right] \notag \\
&= \sum_{\mathcal{I} \subseteq [l]} B^{l-|\mathcal{I}|}
        \sum_{\substack{ s_1, \dots, s_l =1 \\ \text{all distinct} }}^{\tau_N(t)}
        \Bigg( \prod_{i\in\mathcal{I}} c_N(s_i) \Bigg)
        \Bigg( \prod_{j\notin\mathcal{I}} D_N(s_j) \Bigg). \label{eq:010}
\end{align}
Since we are summing over all permutations of $s_{ 1 : l }$, the inner sum depends on $\mathcal{I}$ only through $I:= |\mathcal{I}|$. We may therefore replace the sum over $\mathcal{I} \subseteq [l]$ with a sum over the size $I$ of the subset and a binomial coefficient counting the number of terms in which the subset is of size $I$:
\begin{align}
&\sum_{\mathcal{I} \subseteq [l]} B^{l-|\mathcal{I}|}
        \sum_{\substack{ s_1, \dots, s_l =1 \\ \text{all distinct} }}^{\tau_N(t)}
        \Bigg( \prod_{i\in\mathcal{I}} c_N(s_i) \Bigg)
        \Bigg( \prod_{j\notin\mathcal{I}} D_N(s_j) \Bigg) \notag\\
&= \sum_{\substack{ s_1, \dots, s_l =1 \\ \text{all distinct} }}^{\tau_N(t)}
        \prod_{j=1}^l c_N(s_j) + \sum_{I=0}^{l-1} \binom{l}{I} B^{l-I} 
        \sum_{\substack{ s_1, \dots, s_l =1 \\ \text{all distinct} }}^{\tau_N(t)}
        \Bigg( \prod_{i=1}^I c_N(s_i) \Bigg) \Bigg( \prod_{j=I+1}^l D_N(s_j) \Bigg), \label{eq:012}
\end{align}
where we have also separated out the $I = l$ term.
There is always at least one $D_N$ factor in the second term on the right-hand side, so using Proposition~\ref{thm:cN_properties}\ref{item:cN_property1}, Lemma \ref{thm:sumprod1}\ref{thm:sumprod1_a}, and the Binomial Theorem, we can write
\begin{align}
\sum_{I=0}^{l-1}\binom{l}{I} B^{l-I} 
        &\sum_{\substack{ s_1, \dots, s_l =1 \\ \text{all distinct} }}^{\tau_N(t)}
        \Bigg( \prod_{i=1}^I c_N(s_i) \Bigg) \Bigg( \prod_{j=I+1}^l D_N(s_j) \Bigg) \notag\\
&\leq \sum_{I=0}^{l-1}\binom{l}{I} B^{l-I} 
        \Bigg( \sum_{\substack{ s_1, \dots, s_{l-1} =1 \\ \text{all distinct} }}^{\tau_N(t)}
        \prod_{i=1}^{l-1} c_N(s_i) \Bigg) 
        \sum_{s_l=1}^{\tau_N(t)} D_N(s_l) \notag\\
&\leq \sum_{I=0}^{l-1}\binom{l}{I} B^{l-I} (t+1)^{l-1}
        \sum_{s=1}^{\tau_N(t)} D_N(s) \notag \\
&\leq \Bigg( \sum_{s=1}^{\tau_N(t)} D_N(s) \Bigg) (t+1)^{l-1} (1+B)^l. \label{eq:013}
\end{align}
Substituting \eqref{eq:013} into \eqref{eq:012} concludes the proof.
\end{proof}

\begin{lemma}\label{thm:sumprod3}
Fix $t>0$, $l\in\mathbb{N}$.
Then, for any constant $B>0$,
\begin{align*}
&\sum_{\substack{ s_1, \dots, s_l =1 \\ \text{\normalfont{all distinct}} }}^{\tau_N(t)} 
        \prod_{j=1}^l \left[ c_N(s_j) - B D_N(s_j) \right] \\
&\geq \sum_{\substack{ s_1, \dots, s_l =1 \\ \text{\normalfont{all distinct}} }}^{\tau_N(t)} 
        \prod_{j=1}^l c_N(s_j)
        - \left( \sum_{s=1}^{\tau_N(t)} D_N(s) \right) (t+1)^{l-1} (1+B)^l .
\end{align*}
\end{lemma}

\begin{proof}
A binomial expansion and manipulations as in \eqref{eq:010}--\eqref{eq:012} gives
\begin{align*}
&\sum_{\substack{ s_1, \dots, s_l =1 \\ \text{all distinct} }}^{\tau_N(t)} 
        \prod_{j=1}^l 
        \left[ c_N(s_j) - B D_N(s_j) \right] \\
&\geq \sum_{\substack{ s_1, \dots, s_l =1 \\ \text{all distinct} }}^{\tau_N(t)} 
        \prod_{j=1}^l c_N(s_j) - \sum_{I=0}^{l-1} \binom{l}{I} B^{l-I} 
        \sum_{\substack{ s_1, \dots, s_l =1 \\ \text{all distinct} }}^{\tau_N(t)}
        \Bigg( \prod_{i=1}^I c_N(s_i) \Bigg) \Bigg( \prod_{j=I+1}^l D_N(s_j) \Bigg),
\end{align*}
where the inequality arises because some positive terms have been multiplied by $-1$.
Then \eqref{eq:013} concludes the proof, noting that an upper bound on negative terms results in an overall lower bound.
\end{proof}

Since $l$ distinct objects can always be ordered, Lemmata \ref{thm:sumprod2} and \ref{thm:sumprod3} can also be phrased in terms of summations over ordered rather than distinct variables.
We will use whichever representation is more convenient on a case-by-case basis.

\subsection{Main components of induction argument}\label{sec:main_components_of_induction}
This section contains the technical aspects of the proof of Lemma~\ref{thm:holdingtimes_distn}, which establishes the limiting distributions of holding times of the coupled process via an induction argument.
It is split into four lemmata: the first (Lemma~\ref{thm:basis}) is used in the basis step, and the others in the induction step, which is established by combining upper and lower bounds proved in Lemmata \ref{thm:inductionUB} and \ref{thm:inductionLB}, respectively.
Lemma~\ref{thm:induction_sumprodcN} is a technical result which is common to both the upper and lower bounds, determining the limit as $N\to\infty$ of a certain expectation that arises in both cases.

The following are all consequences of \eqref{eq:mainthmcondition2}: for all $t>s>0$,
\begin{align}
\E \left[ c_N(\tau_N(t)) \right] &\rightarrow 0, \label{eq:BJJK_eq3.3}\\
\E\Bigg[ \sum_{r=\tau_N(s)+1}^{\tau_N(t)} c_N(r)^2 \Bigg] &\rightarrow 0,
        \label{eq:BJJK_eq3.5}\\
\E\Bigg[ \sum_{r=\tau_N(s)+1}^{\tau_N(t)} D_N(r) \Bigg] &\rightarrow 0,
        \label{eq:BJJK_eq3.4}
\end{align}
as $N\to\infty$. Proofs are given in \citet{brown2021} in Lemmata 3.4 (with small tweaks), 3.3 and 3.5 respectively.

\begin{lemma}[Basis step]\label{thm:basis}
Assume \eqref{eq:mainthmcondition2} holds.
Then for any $0 < t < \infty$,
\begin{equation*}
\lim_{N\to\infty} \E\Bigg[ \prod_{r=1}^{\tau_N(t)} (1-p_r) \Bigg]
= e^{-\alpha_n t},
\end{equation*}
where $\alpha_n := n(n-1)/2$.
\end{lemma}

\begin{proof}
We start by showing that
$\lim_{N\to\infty}\E[ \prod_{r=1}^{\tau_N(t)} (1-p_r) ] 
\leq e^{-\alpha_n t}$.\\
Setting $\xi=\Delta$ in Proposition~\ref{thm:pDelta_UB}, we have for each $r$ and sufficiently large $N$,
\begin{equation} \label{eq:018}
1-p_r
= p_{\Delta\Delta}(r) 
\leq 1 - \alpha_n \ON
        \left[ c_N(r) - B_n^\prime D_N(r) \right] .
\end{equation}
Since our interest is the $N\to\infty$ limit, it is sufficient to have bounds that hold for large enough $N$. However, some of the manipulations to follow will also require pre-limiting bounds to be non-negative.
For this reason we introduce indicator functions which guarantee non-negativity, but which will not affect the limit.
The indicators introduced at this point are such that if their conditions do not hold then the bound becomes the trivial $1-p_r \leq 1$.
 
When $N \geq 3$, a sufficient condition to ensure that the expression on the right-hand side of \eqref{eq:018} is non-negative is that the event
\begin{equation}\label{eq:defn_E1}
E_N^{1}(r) := \left\{ c_N(r) < \alpha_n^{-1} \ON \right\} 
\end{equation}
occurs, where the sequence $\ON$ is the same as that in \eqref{eq:018}.
We will also need to control the sign of $c_N(r) - B_n^\prime D_N(r)$, for which we define the event
\begin{equation}\label{eq:defn_E2}
E_N^2(r) := \left\{ c_N(r) \geq B_n^\prime D_N(r) \right\} ,
\end{equation}
and we define $E_N^1 := \bigcap_{r=1}^{\tau_N(t)} E_N^1(r)$ and $E_N^2 := \bigcap_{r=1}^{\tau_N(t)} E_N^2(r)$.
Then
\begin{equation*}
1-p_r
= p_{\Delta\Delta}(r) \leq 1 - \alpha_n \ON 
        \left[ c_N(r) - B_n^\prime D_N(r) \right] \1{E_N^1 \cap E_N^2} .
\end{equation*}
Applying a multinomial expansion and then separating the positive and negative terms,
\begin{align}
\prod_{r=1}^{\tau_N(t)} (1-p_r)
&\leq 1 + \sum_{\substack{l=2 \\ \text{even} }}^{\tau_N(t)} 
        \alpha_n^l \ON \frac{1}{l!} 
        \sum_{\substack{ s_1, \dots, s_l =1 \\ \text{all distinct} }}^{\tau_N(t)}
        \prod_{j=1}^l
        \left[ c_N(s_j) - B_n^\prime D_N(s_j) \right] \1{E_N^1 \cap E_N^2} \notag\\
    &\phantom{\leq 1} - \sum_{\substack{l=1 \\ \text{odd} }}^{\tau_N(t)} 
        \alpha_n^l \ON \frac{1}{l!} 
        \sum_{\substack{ s_1, \dots, s_l =1 \\ \text{all distinct} }}^{\tau_N(t)}
        \prod_{j=1}^l
        \left[ c_N(s_j) - B_n^\prime D_N(s_j) \right] \1{E_N^1 \cap E_N^2} .
        \label{eq:019}
\end{align}
This is further bounded by applying Lemma~\ref{thm:sumprod3} and then both bounds of Lemma~\ref{thm:sumprod1}\ref{thm:sumprod1_b}:
\begin{align*}
&\prod_{r=1}^{\tau_N(t)} (1-p_r) \\
&\leq 1 + \1{E_N^1 \cap E_N^2} \Bigg\{ 
        \sum_{\substack{l=2 \\ \text{even} }}^{\tau_N(t)} 
        \alpha_n^l \ON \frac{1}{l!} 
        \sum_{\substack{ s_1, \dots, s_l =1 \\ \text{all distinct} }}^{\tau_N(t)}
        \prod_{j=1}^l c_N(s_j) \\
    &\phantom{\leq}- \sum_{\substack{l=1 \\ \text{odd} }}^{\tau_N(t)} 
        \alpha_n^l \ON \frac{1}{l!} 
        \Bigg[ \sum_{\substack{ s_1, \dots, s_l =1 \\ \text{all distinct} }}^{\tau_N(t)}
        \prod_{j=1}^l c_N(s_j)
        - \Bigg( \sum_{s=1}^{\tau_N(t)} D_N(s) \Bigg) 
        (t+1)^{l-1} (1+B_n^\prime)^l \Bigg] \Bigg\} \\
&\leq 1 + \1{E_N^1 \cap E_N^2} \Bigg\{ \sum_{\substack{l=2 \\ \text{even} }}^{\tau_N(t)} 
        \alpha_n^l \ON \frac{1}{l!} 
        \left\{ t^l + c_N(\tau_N(t)) (t+1)^l \right\} \\
    &\phantom{\leq 1 + \1{E_N^1 \cap E_N^2} \Bigg\{}- \sum_{\substack{l=1 \\ \text{odd} }}^{\tau_N(t)} 
        \alpha_n^l \ON \frac{1}{l!} 
        \Bigg[ t^l - \Bigg( \sum_{s=1}^{\tau_N(t)} c_N(s)^2 \Bigg) 
        \binom{l}{2} (t+1)^l \Bigg] \\
    &\phantom{\leq 1 + \1{E_N^1 \cap E_N^2} \Bigg\{}- \Bigg( \sum_{s=1}^{\tau_N(t)} D_N(s) \Bigg) 
        (t+1)^{l-1} (1+B_n^\prime)^l \Bigg\}.
\end{align*}
Collecting some terms,
\begin{align}
&\prod_{r=1}^{\tau_N(t)} (1-p_r)
\leq 1 + \sum_{l=1}^{\tau_N(t)} (-\alpha_n)^l \ON \frac{1}{l!} t^l 
        \1{E_N^1 \cap E_N^2}
        + c_N(\tau_N(t)) \sum_{\substack{l=2 \\ \text{even} }}^{\tau_N(t)}
        \alpha_n^l \ON \frac{1}{l!} (t+1)^l \notag\\
    &\phantom{\prod_{r=1}^{\tau_N(t)} (1-p_r) \leq 1} + \Bigg( \sum_{s=1}^{\tau_N(t)} c_N(s)^2 \Bigg)
        \sum_{\substack{l=1 \\ \text{odd} }}^{\tau_N(t)} \alpha_n^l
        \ON \frac{1}{l!} \binom{l}{2} (t+1)^l \notag\\
    &\phantom{\prod_{r=1}^{\tau_N(t)} (1-p_r) \leq 1} + \Bigg( \sum_{s=1}^{\tau_N(t)} D_N(s) \Bigg) 
        \sum_{\substack{l=1 \\ \text{odd} }}^{\tau_N(t)} \alpha_n^l
        \ON \frac{1}{l!} (t+1)^{l-1} (1+B_n^\prime)^l \notag\\
&\phantom{\prod_{r=1}^{\tau_N(t)} (1-p_r)}\leq 1 + \sum_{l=1}^{\infty} (-\alpha_n)^l \ON \frac{1}{l!} t^l
        \I{\tau_N(t) \geq l} \1{E_N^1 \cap E_N^2}      
        + c_N(\tau_N(t)) \exp[ \alpha_n \ON (t+1) ] \notag\\
    &\phantom{\prod_{r=1}^{\tau_N(t)} (1-p_r) \leq 1} + \Bigg( \sum_{s=1}^{\tau_N(t)} c_N(s)^2 \Bigg)
        \frac{1}{2} \alpha_n^2 (t+1)^2 \exp[ \alpha_n \ON (t+1) ] \notag\\
    &\phantom{\prod_{r=1}^{\tau_N(t)} (1-p_r) \leq 1} + \Bigg( \sum_{s=1}^{\tau_N(t)} D_N(s) \Bigg)
        \exp[ \alpha_n \ON (t+1) (1+B_n^\prime) ] . \label{eq:021}
\end{align}
The requirement $\tau_N(t) \geq l$ has been dropped in all but the first term, which constitutes adding some positive terms, giving an upper bound.
Now, taking the expectation and limit, then applying \eqref{eq:BJJK_eq3.3}--\eqref{eq:BJJK_eq3.4}, and using Lemmata \ref{thm:indicators_cN}, \ref{thm:indicators_tau} and \ref{thm:indicators_DN} to show that $\lim_{N\to\infty} \Prob \left[ \{\tau_N(t) \geq l\} \cap E_N^1 \cap E_N^2 \right] =1$,
\begin{align}
\lim_{N\to\infty} \E \Bigg[ \prod_{r=1}^{\tau_N(t)} (1-p_r) \Bigg]
&\leq 1+ \sum_{l=1}^{\infty} (-\alpha_n)^l \frac{1}{l!} t^l
        \lim_{N\to\infty} \Prob \left[ \{\tau_N(t) \geq l\} \cap E_N^1 \cap E_N^2 \right] \notag\\
    &\qquad+ \lim_{N\to\infty} \E \left[ c_N(\tau_N(t)) \right]
        \exp[ \alpha_n (t+1) ] \notag\\
    &\qquad+ \lim_{N\to\infty} \E \Bigg[ \sum_{s=1}^{\tau_N(t)} 
        c_N(s)^2 \Bigg]
        \frac{1}{2} \alpha_n^2 (t+1)^2 \exp[ \alpha_n (t+1) ] \notag\\
    &\qquad+ \lim_{N\to\infty} \E \Bigg[ \sum_{s=1}^{\tau_N(t)} D_N(s) \Bigg]
        \exp[ \alpha_n (t+1) (1+B_n^\prime) ] \notag\\
&= 1+ \sum_{l=1}^{\infty} (-\alpha_n)^l \frac{1}{l!} t^l
= e^{-\alpha_n t}. \label{eq:022}
\end{align}
Passing the limit and expectation inside the infinite sum is justified by dominated convergence and Fubini's theorem.

It remains to show the corresponding lower bound: $\lim_{N\to\infty} \E[ \prod_{r=1}^{\tau_N(t)} (1-p_r) ] 
\geq e^{-\alpha_n t}$.
Setting $\xi=\Delta$ in Proposition~\ref{thm:pDelta_LB}, we have
\begin{equation}\label{eq:pDeltaDelta_LB}
1-p_t
= p_{\Delta\Delta}(t) \geq 1 - \frac{N^{n-2}}{(N-2)_{n-2}} \alpha_n 
    [ c_N(t) + B_n D_N(t) ] ,
\end{equation}
where $B_n >0$.
Due to Proposition~\ref{thm:cN_properties}\ref{item:cN_property1}, a sufficient condition for this bound to be non-negative is
\begin{equation}\label{eq:defn_E3}
E_N^3(r)
:= \left\{ c_N(r) \leq \frac{(N-2)_{n-2}}{N^{n-2}} 
        \alpha_n^{-1} (1+B_n)^{-1} \right\} ,
\end{equation}
and we define $E_N^3 := \bigcap_{r=1}^{\tau_N(t)} E_N^3(r)$. 
Then
\begin{equation*}
1-p_t
\geq \left\{ 1 - \frac{N^{n-2}}{(N-2)_{n-2}} \alpha_n 
    [ c_N(t) + B_n D_N(t) ] \right\} \1{E_N^3(t)}
\end{equation*}
is also a valid lower bound since if $E_N^3$ does not occur then this collapses to the trivial lower bound $1-p_t \geq 0$.
We now apply a multinomial expansion to the product, and split into positive and negative terms:
\begin{align}
\prod_{r=1}^{\tau_N(t)} (1-p_r)
&\geq \Bigg\{ 1 + \sum_{\substack{l=2 \\ \text{even} }}^{\tau_N(t)} 
        \alpha_n^l \ON \frac{1}{l!} 
        \sum_{\substack{ s_1, \dots, s_l =1 \\ \text{all distinct} }}^{\tau_N(t)}
        \prod_{j=1}^l
        \left[ c_N(s_j) + B_n D_N(s_j) \right] \nonumber \\
    &\qquad- \sum_{\substack{l=1 \\ \text{odd} }}^{\tau_N(t)} 
        \alpha_n^l \ON \frac{1}{l!}
        \sum_{\substack{ s_1, \dots, s_l =1 \\ \text{all distinct} }}^{\tau_N(t)}
        \prod_{j=1}^l
        \left[ c_N(s_j) + B_n D_N(s_j) \right] \Bigg\} \1{E_N^3}. \label{eqlabel_1}
\end{align}
From here, the argument for the lower bound follows the same steps as that used to obtain the upper bound.
The right-hand side of \eqref{eqlabel_1} is further bounded via Lemma~\ref{thm:sumprod2} and both bounds in Lemma~\ref{thm:sumprod1}\ref{thm:sumprod1_b}:
\begin{align*}
&\prod_{r=1}^{\tau_N(t)} (1-p_r) \\
&\geq \1{E_N^3} \Bigg\{ 1 + 
        \sum_{\substack{l=2 \\ \text{even} }}^{\tau_N(t)} 
        \alpha_n^l \ON \frac{1}{l!} 
        \sum_{\substack{ s_1, \dots, s_l =1 \\ \text{all distinct} }}^{\tau_N(t)}
        \prod_{j=1}^l c_N(s_j) - \sum_{\substack{l=1 \\ \text{odd} }}^{\tau_N(t)} 
        \alpha_n^l \ON \frac{1}{l!} \\
    &\phantom{\geq}\times
        \Bigg[ \sum_{\substack{ s_1, \dots, s_l =1 \\ \text{all distinct} }}^{\tau_N(t)}
        \prod_{j=1}^l c_N(s_j)
        + \Bigg( \sum_{s=1}^{\tau_N(t)} D_N(s) \Bigg)
        (t+1)^{l-1} (1+B_n)^l \Bigg] \Bigg\} \\
&\geq \1{E_N^3} \Bigg\{ 1 + 
        \sum_{\substack{l=2 \\ \text{even} }}^{\tau_N(t)} 
        \alpha_n^l \ON \frac{1}{l!} 
        \Bigg[ t^l - \Bigg( \sum_{s=1}^{\tau_N(t)} c_N(s)^2 \Bigg)
        \binom{l}{2}(t+1)^l \Bigg] - \sum_{\substack{l=1 \\ \text{odd} }}^{\tau_N(t)} 
        \alpha_n^l \\
    &\phantom{\geq}
        \times \ON \frac{1}{l!} \Bigg[ t^l + c_N(\tau_N(t)) (t+1)^l
        + \Bigg( \sum_{s=1}^{\tau_N(t)} D_N(s) \Bigg)
        (t+1)^{l-1} (1+B_n)^l \Bigg] \Bigg\} .
\end{align*}
Collecting terms and dropping indicators from some non-positive terms,
\begin{align}
\prod_{r=1}^{\tau_N(t)} (1-p_r)
&\geq \sum_{l=0}^{\tau_N(t)} (-\alpha_n)^l \ON 
        \frac{1}{l!} t^l \1{E_N^3} - c_N(\tau_N(t)) \sum_{\substack{l=1 \\ \text{odd} }}^{\tau_N(t)} 
        \alpha_n^l \ON \frac{1}{l!} (t+1)^l \notag\\
        &\qquad- \Bigg( \sum_{s=1}^{\tau_N(t)} c_N(s)^2 \Bigg)
        \sum_{\substack{l=2 \\ \text{even} }}^{\tau_N(t)} 
        \alpha_n^l \ON \frac{1}{l!} \binom{l}{2}(t+1)^l \notag\\
    &\qquad- \Bigg( \sum_{s=1}^{\tau_N(t)} D_N(s) \Bigg)
        \sum_{\substack{l=1 \\ \text{odd} }}^{\tau_N(t)} 
        \alpha_n^l \ON \frac{1}{l!} (t+1)^{l-1} (1+B_n)^l \notag\\
&\geq \sum_{l=0}^{\infty} (-\alpha_n)^l \ON 
        \frac{1}{l!} t^l \1{E_N^3} \I{ \tau_N(t) \geq l} - c_N(\tau_N(t)) \exp[ \alpha_n \ON (t+1) ] \notag \\
    &\qquad - \Bigg( \sum_{s=1}^{\tau_N(t)} c_N(s)^2 \Bigg)
        \frac{1}{2} \alpha_n^2 (t+1)^2 \exp[ \alpha_n \ON (t+1) ]\notag\\
    &\qquad- \Bigg( \sum_{s=1}^{\tau_N(t)} D_N(s) \Bigg)
        \exp[ \alpha_n \ON (t+1) (1+B_n) ]. \label{eq:028}
\end{align}
Now, taking the expectation and limit, and applying \eqref{eq:BJJK_eq3.3}--\eqref{eq:BJJK_eq3.4} to show that all but the first sum vanish, and Lemmata \ref{thm:indicators_cN} and \ref{thm:indicators_tau} to show that $\lim_{N\to\infty} \Prob[ \{\tau_N(t) \geq l\} \cap E_N^3 ] =1$,
\begin{align}
\lim_{N\to\infty} \E \Bigg[ \prod_{r=1}^{\tau_N(t)} (1-p_r) \Bigg]
&\geq \sum_{l=0}^{\infty} (-\alpha_n)^l \ON \frac{1}{l!} t^l 
        \lim_{N\to\infty} \Prob\left[ \{\tau_N(t) \geq l\} \cap E_N^3 \right] \notag\\
    &\qquad- \lim_{N\to\infty} \E \Bigg[ \sum_{s=1}^{\tau_N(t)} c_N(s)^2 \Bigg]
        \frac{1}{2} \alpha_n^2 (t+1)^2 \exp[ \alpha_n (t+1) ]\notag\\
    &\qquad- \lim_{N\to\infty} \E \Bigg[ c_N(\tau_N(t)) \Bigg] 
        \exp[ \alpha_n (t+1) ] \notag\\
    &\qquad- \lim_{N\to\infty} \E \Bigg[ \sum_{s=1}^{\tau_N(t)} D_N(s) \Bigg]
        \exp[ \alpha_n (t+1) (1+B_n) ] \notag\\
&= \sum_{l=0}^{\infty} (-\alpha_n)^l \frac{1}{l!} t^l
= e^{-\alpha_n t}. \label{eq:029}
\end{align}
Again, passing the limit and expectation inside the infinite sum is justified by dominated convergence and Fubini.
Combining the upper and lower bounds in \eqref{eq:022} and \eqref{eq:029} respectively concludes the proof.
\end{proof}

\begin{lemma}[Induction step upper bound]\label{thm:inductionUB}
Assume \eqref{eq:mainthmcondition2} holds.
Fix $k \in \mathbb{N}$, $i_0:=0$, $i_k:=k$. For any sequence of times
$0 = t_0 \leq t_1 \leq \cdots \leq t_k \leq t$,
\begin{align*}
\lim_{N\to\infty} \E \Bigg[ 
        \sum_{\substack{r_1 <\dots< r_k :\\ r_i \leq \tau_N(t_i) \forall i}}
        &\Bigg( \prod_{i=1}^k p_{r_i} \Bigg)
        \Bigg( \prod_{\substack{r=1 \\ \notin \{r_1,\dots,r_k\} }}^{\tau_N(t)} 
        (1-p_r) \Bigg) \Bigg] \\
&\leq \alpha_n^k e^{-\alpha_n t}
        \sum_{\substack{i_1\leq \dots\leq i_{k-1}\\ \in \{0,\dots,k\} :
         i_j \geq j}} %\forall j}} 
        \prod_{j=1}^k \frac{(t_j - t_{j-1})^{i_j - i_{j-1}}}{(i_j - i_{j-1})! } .
\end{align*}
\end{lemma}

\begin{proof}
We use the bound on $(1-p_r)$ from \eqref{eq:018}, which holds for sufficiently large $N$, and apply a multinomial expansion. Define events $E_N^1$ and $E_N^2$ as intersections of events of the form in \eqref{eq:defn_E1} and \eqref{eq:defn_E2}, such that the following manipulations make sense:
\begin{align}
&\prod_{\substack{r=1 \\ \notin \{r_1,\dots,r_k\} }}^{\tau_N(t)} (1-p_r)
\leq \prod_{\substack{r=1 \\ \notin \{r_1,\dots,r_k\} }}^{\tau_N(t)} 
        \left\{ 1 - \alpha_n  \ON [ c_N(r) - B_n^\prime D_N(r) ] 
        \1{E_N^1 \cap E_N^2} \right\} \notag\\
&= 1 + \sum_{l=1}^{\tau_N(t) -k}
        (-\alpha_n)^l \ON \frac{1}{l!}
        \sum_{\substack{s_1, \dots, s_l =1 \\ \notin \{r_1,\dots,r_k\} 
        \\ \text{all distinct} }}^{\tau_N(t)}
        \prod_{j=1}^l [ c_N(s_j) - B_n^\prime D_N(s_j) ]
        \1{E_N^1 \cap E_N^2} \notag\\
&= 1 + \sum_{l=1}^{\tau_N(t) -k}
        (-\alpha_n)^l \ON \frac{1}{l!}
        \sum_{\substack{ s_1, \dots, s_l =1 \\ \text{all distinct} }}^{\tau_N(t)}
        \prod_{j=1}^l [ c_N(s_j) - B_n^\prime D_N(s_j) ] \1{E_N^1 \cap E_N^2} \notag\\
    &\qquad- \sum_{l=1}^{\tau_N(t) -k}
        (-\alpha_n)^l \ON \frac{1}{l!}
        \sum_{\substack{s_1, \dots, s_l =1 \\ \text{all distinct}: 
        \\ \exists i,i^\prime : s_i=r_{i^\prime} }}^{\tau_N(t)} 
        \prod_{j=1}^l [ c_N(s_j) - B_n^\prime D_N(s_j) ]
        \1{E_N^1 \cap E_N^2} . \label{eq:031}
\end{align}
The penultimate line above is exactly the expansion we had in the basis step \eqref{eq:019}, except for the upper limit of the summation over $l$, and as such following the same arguments gives a bound analogous to that in \eqref{eq:021}:
\begin{align*}
&1 + \sum_{l=1}^{\tau_N(t) -k}
        (-\alpha_n)^l \ON \frac{1}{l!}
        \sum_{\substack{ s_1, \dots, s_l =1 \\ \text{all distinct} }}^{\tau_N(t)}
        \prod_{j=1}^l [ c_N(s_j) - B_n^\prime D_N(s_j) ] \1{E_N^1 \cap E_N^2} \\
&\leq 1+ \sum_{l=1}^{\tau_N(t) -k} (-\alpha_n)^l \ON \frac{1}{l!} t^l
        \1{E_N^1 \cap E_N^2}
        + c_N(\tau_N(t)) \exp[ \alpha_n \ON (t+1) ] \\
    &\qquad+ \Bigg( \sum_{s=1}^{\tau_N(t)} c_N(s)^2 \Bigg)
        \frac{1}{2} \alpha_n^2 (t+1)^2 \exp[ \alpha_n \ON (t+1) ] \\
    &\qquad+ \Bigg( \sum_{s=1}^{\tau_N(t)} D_N(s) \Bigg)
        \exp[ \alpha_n \ON (t+1) (1+B_n^\prime) ] .
\end{align*}
For the last line of \eqref{eq:031}, recalling that $D_N(t) \leq c_N(t)$ (Proposition~\ref{thm:cN_properties}\ref{item:cN_property1}),
\begin{align*}
&- \sum_{l=1}^{\tau_N(t) -k} (-\alpha_n)^l \ON \frac{1}{l!}
        \sum_{\substack{s_1, \dots, s_l =1 \\ \text{all distinct}: 
        \\ \exists i,i^\prime : s_i=r_{i^\prime} }}^{\tau_N(t)} 
        \prod_{j=1}^l \{ c_N(s_j) - B_n^\prime D_N(s_j) \} 
        \1{E_N^1 \cap E_N^2} \\
&\leq \sum_{l=1}^{\tau_N(t) -k} \alpha_n^l \ON \frac{1}{l!}
        \sum_{\substack{s_1, \dots, s_l =1 \\ \text{all distinct}: 
        \\ \exists i,i^\prime : s_i=r_{i^\prime} }}^{\tau_N(t)} 
        \prod_{j=1}^l \{ c_N(s_j) + B_n^\prime D_N(s_j) \} \\
&\leq \sum_{l=1}^{\tau_N(t) -k} \alpha_n^l \ON \frac{1}{l!}
        \sum_{\substack{s_1, \dots, s_l =1 \\ \text{all distinct}: 
        \\ \exists i,i^\prime : s_i=r_{i^\prime} }}^{\tau_N(t)} 
        (1 + B_n^\prime)^l \prod_{j=1}^l c_N(s_j) \\
&\leq \sum_{s \in \{r_1,\dots,r_k\} } c_N(s)
        \sum_{l=1}^{\tau_N(t) -k} \alpha_n^l \ON
        \frac{1}{(l-1)!}  (1 + B_n^\prime)^l
        \sum_{\substack{ s_1, \dots, s_{l-1} =1 \\ \text{all distinct} }}^{\tau_N(t)}
        \prod_{j=1}^{l-1} c_N(s_j) \\
&\leq \sum_{j=1}^k c_N(r_j)
        \sum_{l=1}^{\tau_N(t) -k} \alpha_n^l \ON
        \frac{1}{(l-1)!}  (1 + B_n^\prime)^l (t+1)^{l-1} \\
&\leq \left( \sum_{j=1}^k c_N(r_j) \right)
        \alpha_n (1 + B_n^\prime) 
        \exp[ \alpha_n \ON (1 + B_n^\prime) (t+1) ] ,
\end{align*}
where the penultimate inequality uses Lemma~\ref{thm:sumprod1}\ref{thm:sumprod1_a}.
Substituting the preceding two displays into \eqref{eq:031}, we obtain
\begin{align}
\prod_{\substack{r=1 \\ \notin \{r_1,\dots,r_k\} }}^{\tau_N(t)} (1-p_r)&\leq 1+ \sum_{l=1}^{\tau_N(t) -k} (-\alpha_n)^l \ON \frac{1}{l!} t^l
        \1{E_N^1 \cap E_N^2}
        + c_N(\tau_N(t)) \exp[ \alpha_n \ON (t+1) ] \notag\\
    &\qquad+ \left( \sum_{s=1}^{\tau_N(t)} c_N(s)^2 \right)
        \frac{1}{2} \alpha_n^2 (t+1)^2 \exp[ \alpha_n \ON (t+1) ] \notag\\
    &\qquad+ \left( \sum_{s=1}^{\tau_N(t)} D_N(s) \right)
        \exp[ \alpha_n \ON (t+1) (1+B_n^\prime) ] \notag\\
    &\qquad+ \left( \sum_{j=1}^k c_N(r_j) \right)
        \alpha_n (1 + B_n^\prime)
        \exp[ \alpha_n \ON (1 + B_n^\prime) (t+1) ] . \label{eq:034b}
\end{align}
To obtain a corresponding bound for $p_r$, we use \eqref{eq:pDeltaDelta_LB} and Lemma~\ref{thm:sumprod2} (with ordered rather than distinct indices) to obtain
\begin{align}
\sum_{\substack{r_1 <\dots< r_k :\\ r_i \leq \tau_N(t_i) \forall i}}
        \prod_{i=1}^k p_{r_i}
&\leq \alpha_n^k \ON 
        \sum_{\substack{r_1 <\dots< r_k :\\ r_i \leq \tau_N(t_i) \forall i}}
        \prod_{i=1}^k \left[ c_N(r_i) + B_n D_N(r_i) \right] \notag\\
&\leq \alpha_n^k \ON
        \sum_{\substack{r_1 <\dots< r_k :\\ r_i \leq \tau_N(t_i) \forall i}}
        \prod_{i=1}^k c_N(r_i) \notag \\
        &\phantom{\leq} + \Bigg( \sum_{s=1}^{\tau_N(t)} D_N(s) \Bigg)
        \alpha_n^k \ON (t+1)^{k-1} (1+B_n)^k . \label{eq:035}
\end{align}
The following looser but simpler bound will also be useful:
\begin{align}
\prod_{i=1}^k p_{r_i}
\leq \alpha_n^k \ON 
        \prod_{i=1}^k \left[ c_N(r_i) + B_n D_N(r_i) \right] 
        &\leq \alpha_n^k \ON  \prod_{i=1}^k c_N(r_i) (1 + B_n) \notag\\
&\leq \alpha_n^k \ON (1 + B_n)^k
        \prod_{i=1}^k c_N(r_i). \label{eq:036}
\end{align}
Using Lemma~\ref{thm:sumprod1}\ref{thm:sumprod1_a}, \eqref{eq:036} also leads to the deterministic bound
\begin{align}
\sum_{\substack{r_1 <\dots< r_k :\\ r_i \leq \tau_N(t_i) \forall i}}
        \prod_{i=1}^k p_{r_i}
&\leq \alpha_n^k \ON (1 + B_n)^k \frac{1}{k!}
        \sum_{r_1 \neq \dots\neq r_k}^{\tau_N(t)} 
        \prod_{i=1}^k c_N(r_i) \notag\\
&\leq \alpha_n^k \ON (1 + B_n)^k \frac{1}{k!} (t+1)^k. 
        \label{eq:037}
\end{align}
All the ingredients for obtaining the bound in the statement of Lemma \ref{thm:inductionUB} are now in place.
First, by \eqref{eq:034b},
\begin{align*}
&\sum_{\substack{r_1 <\dots< r_k :\\ r_i \leq \tau_N(t_i) \forall i}}
        \Bigg( \prod_{i=1}^k p_{r_i} \Bigg)
        \Bigg( \prod_{\substack{r=1 \\ \notin \{r_1,\dots,r_k\} }}^{\tau_N(t)} 
        (1-p_r) \Bigg) \\
&\leq \Bigg\{ 1+ \sum_{l=1}^{\tau_N(t) -k} (-\alpha_n)^l \ON 
        \frac{1}{l!} t^l \1{E_N^1 \cap E_N^2} \Bigg\}
        \sum_{\substack{r_1 <\dots< r_k :\\ r_i \leq \tau_N(t_i) \forall i}}
        \prod_{i=1}^k p_{r_i} \\
    &\phantom{\leq} + \Bigg\{ c_N(\tau_N(t)) 
        \exp[ \alpha_n \ON (t+1) ] \\
    &\phantom{\leq + \Bigg\{} + \Bigg( \sum_{s=1}^{\tau_N(t)} c_N(s)^2 \Bigg)
        \frac{1}{2} \alpha_n^2 (t+1)^2 \exp[ \alpha_n \ON (t+1) ] \\
    &\phantom{\leq + \Bigg\{} + \Bigg( \sum_{s=1}^{\tau_N(t)} D_N(s) \Bigg)
        \exp[ \alpha_n \ON (t+1) (1+B_n^\prime) ] \Bigg\}
        \sum_{\substack{r_1 <\dots< r_k :\\ r_i \leq \tau_N(t_i) \forall i}}
        \prod_{i=1}^k p_{r_i} \\
    &\phantom{\leq} + \exp[ \alpha_n \ON (1 + B_n^\prime) (t+1) ]
        \alpha_n (1 + B_n^\prime)
        \sum_{\substack{r_1 <\dots< r_k :\\ r_i \leq \tau_N(t_i) \forall i}}
        \sum_{j=1}^k c_N(r_j)
        \prod_{i=1}^k p_{r_i} .
\end{align*}
To further bound the right-hand side, we apply \eqref{eq:035} to the first term, \eqref{eq:037} to the second, and \eqref{eq:036} to the third, yielding
\begin{align*}
&\sum_{\substack{r_1 <\dots< r_k :\\ r_i \leq \tau_N(t_i) \forall i}}
        \Bigg( \prod_{i=1}^k p_{r_i} \Bigg)
        \Bigg( \prod_{\substack{r=1 \\ \notin \{r_1,\dots,r_k\} }}^{\tau_N(t)} 
        (1-p_r) \Bigg) \\
&\leq \alpha_n^k \ON \Bigg\{
        1+ \sum_{l=1}^{\tau_N(t) -k} (-\alpha_n)^l \ON \frac{1}{l!} t^l
        \1{E_N^1 \cap E_N^2} \Bigg\}
        \sum_{\substack{r_1 <\dots< r_k :\\ r_i \leq \tau_N(t_i) \forall i}}
        \prod_{i=1}^k c_N(r_i) \\
    &\phantom{\leq}+ \Bigg( \sum_{s=1}^{\tau_N(t)} D_N(s) \Bigg)
        \alpha_n^k \ON (t+1)^{k-1} (1+B_n)^k
        \sum_{l=0}^{\tau_N(t)} (\alpha_n)^l \ON \frac{1}{l!} t^l \\
    &\phantom{\leq}+ \Bigg\{ c_N(\tau_N(t)) 
        \exp[ \alpha_n \ON (t+1) ]  + \Bigg( \sum_{s=1}^{\tau_N(t)} c_N(s)^2 \Bigg)
        \frac{1}{2} \alpha_n^2 (t+1)^2 \exp[ \alpha_n \ON (t+1) ] \\
    &\phantom{\leq + \Bigg\{}+ \Bigg( \sum_{s=1}^{\tau_N(t)} D_N(s) \Bigg)
        \exp[ \alpha_n \ON (t+1) (1+B_n^\prime) ] \Bigg\}
        \alpha_n^k \ON (1 + B_n)^k \frac{(t+1)^k}{k!} \\
    &\phantom{\leq}+ \exp[ \alpha_n (1 + B_n^\prime) (t+1) ]
        \alpha_n (1 + B_n^\prime)
        \alpha_n^k \ON (1 + B_n)^k \\
    &\hspace{6cm}\times \sum_{\substack{r_1 <\dots< r_k 
        :\\ r_i \leq \tau_N(t_i) \forall i}}
        \sum_{j=1}^k c_N(r_j)
        \prod_{i=1}^k c_N(r_i) .
\end{align*}
Taking an expectation and letting $N \to \infty$, the second, third, fourth, and fifth lines on the right-hand side vanish by \eqref{eq:BJJK_eq3.3}--\eqref{eq:BJJK_eq3.4}, leaving
\begin{align}
&\lim_{N\to\infty} \E \Bigg[ 
        \sum_{\substack{r_1 <\dots< r_k :\\ r_i \leq \tau_N(t_i) \forall i}}
        \Bigg( \prod_{i=1}^k p_{r_i} \Bigg)
        \Bigg( \prod_{\substack{r=1 \\ \notin \{r_1,\dots,r_k\} }}^{\tau_N(t)} 
        (1-p_r) \Bigg) \Bigg] \notag \\
&\leq \alpha_n^k \lim_{N\to\infty} \E \Bigg[
        \sum_{\substack{r_1 <\dots< r_k :\\ r_i \leq \tau_N(t_i) \forall i}}
        \prod_{i=1}^k c_N(r_i) \Bigg] \notag\\
    &\phantom{\leq} + \alpha_n^k
        \sum_{l=1}^{\infty} (-\alpha_n)^l \frac{1}{l!} t^l
        \lim_{N\to\infty} \E \Bigg[ \I{\tau_N(t)\geq k+l} \1{E_N^1 \cap E_N^2}
        \sum_{\substack{r_1 <\dots< r_k :\\ r_i \leq \tau_N(t_i) \forall i}}
        \prod_{i=1}^k c_N(r_i) \Bigg] \notag\\
    &\phantom{\leq}+ \exp[ \alpha_n (1 + B_n^\prime) (t+1) ]
        \alpha_n^{k+1} (1 + B_n^\prime) (1 + B_n)^k \notag\\
    &\phantom{\leq +} \times \lim_{N\to\infty} \E \Bigg[ 
        \sum_{\substack{r_1 <\dots< r_k :\\ r_i \leq \tau_N(t_i) \forall i}}
        \sum_{j=1}^k c_N(r_j)
        \prod_{i=1}^k c_N(r_i) \Bigg] , \label{eq:040}
\end{align}
where passing the limit and expectation inside the infinite sum is justified by Lemma~\ref{thm:DCT_Fubini}.
To see that the last line vanishes, recall that $0 = t_0 \leq t_1 \leq \ldots \leq t_k \leq t$, whereupon
\begin{align*}
\sum_{\substack{r_1 <\dots< r_k :\\ r_i \leq \tau_N(t_i) \forall i}}
        \sum_{j=1}^k c_N(r_j)
        \prod_{i=1}^k c_N(r_i)
&\leq \frac{1}{k!} \sum_{\substack{r_1, \dots, r_k \\ \text{all distinct} }}^{\tau_N(t)}
        \sum_{j=1}^k c_N(r_j)
        \prod_{i=1}^k c_N(r_i) \\
&\leq \frac{1}{k!} 
        \sum_{j=1}^k \sum_{s=1}^{\tau_N(t)} c_N(s)^2
        \sum_{\substack{r_1, \dots, r_{k-1} \\ \text{all distinct} }}^{\tau_N(t)}
        \prod_{i=1}^{k-1} c_N(r_i) \\
&\leq \frac{1}{(k-1)!} 
        \sum_{s=1}^{\tau_N(t)} c_N(s)^2
        (t+1)^{k-1},
\end{align*}
using Lemma~\ref{thm:sumprod1}\ref{thm:sumprod1_a} for the final inequality.
Hence, by \eqref{eq:BJJK_eq3.5},
\begin{align*}
\lim_{N\to\infty} \E &\Bigg[ \sum_{\substack{r_1 <\dots< r_k :\\ r_i \leq \tau_N(t_i) \forall i}}
        \sum_{s \in \{r_1,\dots,r_k\} } c_N(s) \prod_{i=1}^k c_N(r_i) \Bigg] \\
&\leq \frac{1}{(k-1)!} (t+1)^{k-1}
         \lim_{N\to\infty} \E \Bigg[ \sum_{s=1}^{\tau_N(t)} c_N(s)^2 \Bigg] = 0.
\end{align*}
By Lemmata \ref{thm:indicators_cN}, \ref{thm:indicators_tau} and \ref{thm:indicators_DN}, $\lim_{N\to\infty} \Prob[ \{ \tau_N(t)\geq k+l \} \cap E_N^1 \cap E_N^2 ] =1$,  so we can apply Lemma~\ref{thm:induction_sumprodcN} to the remaining expectations in \eqref{eq:040}, yielding
\begin{align*}
&\lim_{N\to\infty} \E \Bigg[ 
        \sum_{\substack{r_1 <\dots< r_k :\\ r_i \leq \tau_N(t_i) \forall i}}
        \Bigg( \prod_{i=1}^k p_{r_i} \Bigg)
        \Bigg( \prod_{\substack{r=1 \\ \notin \{r_1,\dots,r_k\} }}^{\tau_N(t)} 
        (1-p_r) \Bigg) \Bigg] \\
&\hspace{3cm}\leq \alpha_n^k
        \sum_{l=0}^{\infty} (-\alpha_n)^l \frac{1}{l!} t^l
        \sum_{\substack{i_1\leq \dots\leq i_{k-1}\\ \in \{0,\dots,k\} :
        i_j \geq j}}
        \prod_{j=1}^k \frac{(t_j - t_{j-1})^{i_j - i_{j-1}}}{(i_j - i_{j-1})! } \\
&\hspace{3cm}= \alpha_n^k e^{-\alpha_n t}
        \sum_{\substack{i_1\leq \dots\leq i_{k-1}\\ \in \{0,\dots,k\} :
        i_j \geq j}}
        \prod_{j=1}^k \frac{(t_j - t_{j-1})^{i_j - i_{j-1}}}{(i_j - i_{j-1})! }.
\end{align*}
\end{proof}

\begin{lemma}[Induction step lower bound]\label{thm:inductionLB}
Assume \eqref{eq:mainthmcondition2} holds.
Fix $k \in \mathbb{N}$, $i_0:=0$, $i_k:=k$. For any sequence of times
$0 = t_0 \leq t_1 \leq \cdots \leq t_k \leq t$,
\begin{align*}
\lim_{N\to\infty} \E \Bigg[ 
        \sum_{\substack{r_1 <\dots< r_k :\\ r_i \leq \tau_N(t_i) \forall i}}
        &\Bigg( \prod_{i=1}^k p_{r_i} \Bigg)
        \Bigg( \prod_{\substack{r=1 \\ \notin \{r_1,\dots,r_k\} }}^{\tau_N(t)} 
        (1-p_r) \Bigg) \Bigg] \\
&\geq \alpha_n^k e^{-\alpha_n t}
        \sum_{\substack{i_1\leq \dots\leq i_{k-1}\\ \in \{0,\dots,k\} :
        i_j \geq j}}
        \prod_{j=1}^k \frac{(t_j - t_{j-1})^{i_j - i_{j-1}}}{(i_j - i_{j-1})! } .
\end{align*}
\end{lemma}

\begin{proof}
Firstly,
\begin{align}
\sum_{\substack{r_1 <\dots< r_k :\\ r_i \leq \tau_N(t_i) \forall i}}
        &\Bigg( \prod_{i=1}^k p_{r_i} \Bigg)
        \Bigg( \prod_{\substack{r=1 \\ \notin \{r_1,\dots,r_k\} }}^{\tau_N(t)} 
        (1-p_r) \Bigg) \geq \sum_{\substack{r_1 <\dots< r_k :\\ r_i \leq \tau_N(t_i) \forall i}}
        \Bigg( \prod_{i=1}^k p_{r_i} \Bigg)
        \Bigg( \prod_{r=1}^{\tau_N(t)} 
        (1-p_r) \Bigg). \label{eq:032a}
\end{align}
The second product on the right-hand side does not depend on $r_{ 1 : k }$, and we can use the lower bound from \eqref{eq:028}:
\begin{align}
\prod_{r=1}^{\tau_N(t)} (1-p_r)
&\geq \sum_{l=0}^{\tau_N(t)} (-\alpha_n)^l \ON 
        \frac{1}{l!} t^l \1{E_N^3} - c_N(\tau_N(t)) \exp[ \alpha_n \ON (t+1) ] \notag \\
    &\phantom{\geq} - \Bigg( \sum_{s=1}^{\tau_N(t)} c_N(s)^2 \Bigg)
        \frac{1}{2} \alpha_n^2 (t+1)^2 \exp[ \alpha_n \ON (t+1) ]\notag\\
    &\phantom{\geq} - \Bigg( \sum_{s=1}^{\tau_N(t)} D_N(s) \Bigg)
        \exp[ \alpha_n \ON (t+1) (1+B_n) ], \label{eq:033a}
\end{align}
where $E_N^3$ is defined as in and immediately beneath \eqref{eq:defn_E3}.
We will also need an upper bound on this product, which is formed from \eqref{eq:021} with a further deterministic bound:
\begin{align}
\prod_{r=1}^{\tau_N(t)}  (1-p_r)
\leq & \sum_{l=0}^{\tau_N(t)} (-\alpha_n)^l \ON \frac{1}{l!} t^l \1{E_N^1 \cap E_N^2}
        + c_N(\tau_N(t)) \exp[ \alpha_n \ON (t+1) ] \notag\\
    & + \Bigg( \sum_{s=1}^{\tau_N(t)} c_N(s)^2 \Bigg)
        \frac{1}{2} \alpha_n^2 (t+1)^2 \exp[ \alpha_n \ON (t+1) ] \notag\\
    & + \Bigg( \sum_{s=1}^{\tau_N(t)} D_N(s) \Bigg)
        \exp[ \alpha_n \ON (t+1) (1+B_n^\prime) ] \notag\\
\leq& \exp[ \alpha_n \ON t ]
        + \exp[ \alpha_n \ON (t+1) ]  + \frac{1}{2} \alpha_n^2 (t+1)^3
        \exp[ \alpha_n \ON (t+1) ] \notag\\
        &+ (t+1) \exp[ \alpha_n \ON (t+1) (1+B_n^\prime) ] \notag\\
\leq& \left( 2 + \frac{\alpha_n^2 (t+1)^3}{2} \right) 
        \exp[ \alpha_n \ON (t+1) ] \notag \\
    & + (t+1) \exp[ \alpha_n \ON (t+1) (1+B_n^\prime) ],
        \label{eq:034a}
\end{align}
where the second inequality uses Proposition \ref{thm:cN_properties}, parts \ref{item:cN_property1} and \ref{item:cN_property4}.
Now consider the remaining sum-product of $p_{r_i}$-factors on the right-hand side of \eqref{eq:032a}.
We use the same bound on $p_r$ as in \eqref{eq:018}:
\begin{equation}\label{eq:050a}
p_r
= 1 - p_{\Delta\Delta}(r) 
\geq \alpha_n \ON 
        \left[ c_N(r) - B_n^\prime D_N(r) \right], 
\end{equation}
where the $\ON$ term does not depend on $r$.
The right-hand side of \eqref{eq:050a} is non-negative on the event $E_N^2$, defined in and beneath \eqref{eq:defn_E2}.
Hence
\begin{equation*}
\prod_{i=1}^k p_{r_i}
\geq \alpha_n^k \ON 
        \prod_{i=1}^k \left[ c_N(r_i) - B_n^\prime D_N(r_i) \right] \1{E_N^2} .
\end{equation*}
Applying Lemma~\ref{thm:sumprod3} with ordered indices, we obtain
\begin{align*}
\sum_{\substack{r_1 <\dots< r_k :\\ r_i \leq \tau_N(t_i) \forall i}}
        \prod_{i=1}^k p_{r_i}
&\geq \alpha_n^k \ON
        \sum_{\substack{r_1 <\dots< r_k :\\ r_i \leq \tau_N(t_i) \forall i}}
        \prod_{i=1}^k c_N(r_i) \1{E_N^2} \\
    &\phantom{\geq}- \alpha_n^k \ON \frac{1}{k!} \Bigg( \sum_{s=1}^{\tau_N(t)} D_N(s) \Bigg)
        (t+1)^{k-1} (1+B_n^\prime)^k.
\end{align*}
The above expression is already split into positive and negative terms; a lower bound on \eqref{eq:032a} can be formed by multiplying the positive terms by the lower bound \eqref{eq:033a} and the negative terms by the upper bound \eqref{eq:034a}. 
Thus,
\begin{align*}
&\sum_{\substack{r_1 <\dots< r_k :\\ r_i \leq \tau_N(t_i) \forall i}}
        \Bigg( \prod_{i=1}^k p_{r_i} \Bigg)
        \Bigg( \prod_{\substack{r=1 \\ \notin \{r_1,\dots,r_k\} }}^{\tau_N(t)} 
        (1-p_r) \Bigg) \\
&\geq \alpha_n^k \ON
        \sum_{\substack{r_1 <\dots< r_k :\\ r_i \leq \tau_N(t_i) \forall i}}
        \prod_{i=1}^k c_N(r_i) \1{E_N^2} \Bigg\{
        \sum_{l=0}^{\tau_N(t)} (-\alpha_n)^l \ON 
        \frac{1}{l!} t^l \1{E_N^3} \\
    &\phantom{\geq \times} - \Bigg[ \Bigg( \sum_{s=1}^{\tau_N(t)} c_N(s)^2 \Bigg)
        \frac{\alpha_n^2 (t+1)^2}{2} + c_N(\tau_N(t)) \Bigg] \exp[ \alpha_n \ON (t+1) ] \\
    &\phantom{\geq \times}
    - \Bigg( \sum_{s=1}^{\tau_N(t)} D_N(s) \Bigg)
        \exp[ \alpha_n \ON (t+1) (1+B_n) ] \Bigg\} \\
&\phantom{\geq} - \Bigg( \sum_{s=1}^{\tau_N(t)} D_N(s) \Bigg)
        \alpha_n^k \ON \frac{1}{k!}
        (t+1)^{k-1} (1+B_n^\prime)^k \\
    &\phantom{\geq -} \times \Bigg\{\left( 2 + \frac{\alpha_n^2 (t+1)^3}{2} \right) 
        \exp[ \alpha_n \ON (t+1) ] + (t+1) \exp[ \alpha_n \ON (t+1) (1+B_n^\prime) ] 
        \Bigg\} .
\end{align*}
Due to \eqref{eq:BJJK_eq3.3}--\eqref{eq:BJJK_eq3.4}, all but the first line on the right-hand side of the above have vanishing expectation, leaving
\begin{align}
&\lim_{N\to\infty} \E \Bigg[ 
        \sum_{\substack{r_1 <\dots< r_k :\\ r_i \leq \tau_N(t_i) \forall i}}
        \Bigg( \prod_{i=1}^k p_{r_i} \Bigg)
        \Bigg( \prod_{\substack{r=1 \\ \notin \{r_1,\dots,r_k\} }}^{\tau_N(t)} 
        (1-p_r) \Bigg) \Bigg] \notag\\
&\geq \lim_{N\to\infty} \E \Bigg[ \alpha_n^k \ON
        \sum_{\substack{r_1 <\dots< r_k :\\ r_i \leq \tau_N(t_i) \forall i}}
        \prod_{i=1}^k c_N(r_i) \1{E_N^2}
        \sum_{l=0}^{\tau_N(t)} (-\alpha_n)^l \ON 
        \frac{1}{l!} t^l \1{E_N^3} \Bigg] \notag\\
&= \alpha_n^k
        \sum_{l=0}^{\infty} (-\alpha_n)^l
        \frac{1}{l!} t^l
        \lim_{N\to\infty}\E\Bigg[ \I{\tau_N(t) \geq l} \1{E_N^2\cap E_N^3}
        \sum_{\substack{r_1 <\dots< r_k :\\ r_i \leq \tau_N(t_i) \forall i}}
        \prod_{i=1}^k c_N(r_i) \Bigg] . \label{eq:056}
\end{align}
Passing the limit and expectation inside the infinite sum is justified by Lemma~\ref{thm:DCT_Fubini}.
Lemmata \ref{thm:indicators_cN} and \ref{thm:indicators_DN} establish that $\lim_{N\to\infty}\Prob[ E_N^2 \cap E_N^3 ] =1$, and Lemma~\ref{thm:indicators_tau} deals with the indicator for $\{ \tau_N( t ) \geq l \}$.
We can therefore apply Lemma~\ref{thm:induction_sumprodcN} to conclude that
\begin{align*}
\lim_{N\to\infty} \E \Bigg[ 
        \sum_{\substack{r_1 <\dots< r_k :\\ r_i \leq \tau_N(t_i) \forall i}}
        &\Bigg( \prod_{i=1}^k p_{r_i} \Bigg)
        \Bigg( \prod_{\substack{r=1 \\ \notin \{r_1,\dots,r_k\} }}^{\tau_N(t)} 
        (1-p_r) \Bigg) \Bigg] \\
&\geq \alpha_n^k
        \sum_{l=0}^{\infty} (-\alpha_n)^l
        \frac{1}{l!} t^l
        \sum_{\substack{i_1\leq \dots\leq i_{k-1}\\ \in \{0,\dots,k\} :
        i_j \geq j}}
        \prod_{j=1}^k \frac{(t_j - t_{j-1})^{i_j - i_{j-1}}}{(i_j - i_{j-1})! } \\
&= \alpha_n^k e^{-\alpha_n t} 
        \sum_{\substack{i_1\leq \dots\leq i_{k-1}\\ \in \{0,\dots,k\} :
        i_j \geq j}}
        \prod_{j=1}^k \frac{(t_j - t_{j-1})^{i_j - i_{j-1}}}{(i_j - i_{j-1})! },
\end{align*}
as required.
\end{proof}

\begin{lemma}\label{thm:induction_sumprodcN}
Assume \eqref{eq:mainthmcondition2} holds.
Fix $k \in \mathbb{N}$, $i_0:=0$, $i_k:=k$. 
Let $E_N$ be a sequence of events such that 
$\lim_{N\to\infty} \Prob[E_N] =1$. 
Then for any sequence of times 
$0 = t_0 \leq t_1 \leq \cdots \leq t_k \leq t$,
\begin{equation*}
\lim_{N\to\infty} \E \Bigg[ \1{E_N} 
        \sum_{\substack{r_1<\dots<r_k :\\ r_i\leq \tau_N(t_i) \forall i}}
        \prod_{i=1}^k c_N(r_i) \Bigg] 
= \sum_{\substack{i_1\leq \dots\leq i_{k-1}\\ \in \{0,\dots,k\} :
        i_j \geq j}}
        \prod_{j=1}^k \frac{(t_j - t_{j-1})^{i_j - i_{j-1}}}{(i_j - i_{j-1})! } .
\end{equation*}
\end{lemma}

\begin{proof}
As pointed out by \citet[p.\ 460]{mohle1999}, the sum-product on the left-hand side of the statement can be expanded as
\begin{equation*}
\sum_{\substack{r_1<\dots<r_k :\\ r_i\leq \tau_N(t_i) \forall i}} 
        \prod_{i=1}^k c_N(r_i)
= \sum_{\substack{i_1\leq \dots\leq i_{k-1}\\ \in \{0,\dots,k\} :
        i_j \geq j}} \, \prod_{j=1}^k \frac{1}{(i_j - i_{j-1})!}
        \sum_{\substack{ r_{i_{j-1}+1}, \ldots, r_{i_j} \\ = \tau_N(t_{j-1})+1 
        \\ \text{all distinct} }}^{\tau_N(t_j)}  
        \,\prod_{i=i_{j-1}+1}^{i_j} c_N(r_i) .
\end{equation*}
By Lemma~\ref{thm:sumprod1}\ref{thm:sumprod1_b},
\begin{align*}
\sum_{\substack{ r_{i_{j-1}+1}, \ldots, r_{i_j} = \tau_N(t_{j-1})+1 
        \\ \text{all distinct} }}^{\tau_N(t_j)}  
        \prod_{i=i_{j-1}+1}^{i_j} c_N(r_i) \leq{}& (t_j - t_{j-1})^{i_j - i_{j-1}} 
        \\& + c_N(\tau_N(t_j)) ( t_j +1 )^{i_j - i_{j-1}}.
\end{align*}
Hence, a $k$-fold product of similar terms can be bounded as
\begin{align*}
&\prod_{j=1}^k \frac{1}{(i_j - i_{j-1})!}
        \sum_{\substack{ r_{i_{j-1}+1}, \ldots, r_{i_j} = \tau_N(t_{j-1})+1 
        \\ \text{all distinct} }}^{\tau_N(t_j)}
        \,\prod_{i=i_{j-1}+1}^{i_j} c_N(r_i) \\
&\leq \prod_{j=1}^k \Bigg\{ \frac{(t_j - t_{j-1})^{i_j - i_{j-1}}}{(i_j - i_{j-1})!} + c_N(\tau_N(t_j)) \frac{( t_j +1 )^{i_j - i_{j-1}}}{(i_j - i_{j-1})!} \Bigg\} 
        \\
&\leq \prod_{j=1}^k \frac{(t_j - t_{j-1})^{i_j - i_{j-1}}}{(i_j - i_{j-1})!}  + \sum_{\mathcal{I} \subset [k]} \left( \prod_{j\in\mathcal{I}} 
        t^{i_j - i_{j-1}}\right)
        \Bigg( \prod_{j\notin\mathcal{I}} c_N(\tau_N(t_j))
        (t+1)^{i_j - i_{j-1}} \Bigg) \\
&\leq  \prod_{j=1}^k \frac{(t_j - t_{j-1})^{i_j - i_{j-1}}}{(i_j - i_{j-1})!}
        + \sum_{\mathcal{I} \subset [k]} c_N(\tau_N(t_{j^\star(\mathcal{I})})) 
        ( t+1 )^k,
\end{align*}
where, say, $j^\star(\mathcal{I}) := \min\{ j\notin\mathcal{I}\}$, and the second-to-last line follows by separating the $I = [k]$ term, and the last via Proposition \ref{thm:cN_properties}\ref{item:cN_property1}, recalling that $t_0 = 0$ and $t = t_k$.
Now we are in a position to evaluate the desired limit: 
\begin{align*}
&\lim_{N\to\infty} \E \Bigg[ \1{E_N} 
        \sum_{\substack{r_1<\dots<r_k :\\ r_i\leq \tau_N(t_i) \forall i}} 
        \prod_{i=1}^k c_N(r_i) \Bigg] 
\leq \lim_{N\to\infty} \E \Bigg[\sum_{\substack{r_1<\dots<r_k :\\ 
        r_i\leq \tau_N(t_i) \forall i}} \prod_{i=1}^k c_N(r_i) \Bigg] \\
&\leq \sum_{\substack{i_1\leq \dots\leq i_{k-1}\\ \in \{0,\dots,k\} :
        i_j \geq j}} \Bigg\{ \prod_{j=1}^k \frac{(t_j - t_{j-1})^{i_j - i_{j-1}}}
        {(i_j - i_{j-1})!} + \sum_{\mathcal{I} \subset [k]} \lim_{N\to\infty} 
        \E [ c_N(\tau_N(t_{j^\star(\mathcal{I})})) ] 
        (t+1)^k \Bigg\} \\
&= \sum_{\substack{i_1\leq \dots\leq i_{k-1}\\ \in \{0,\dots,k\} : 
        i_j \geq j \forall j}} \prod_{j=1}^k \frac{(t_j - t_{j-1})^{i_j - i_{j-1}}}
        {(i_j - i_{j-1})!},
\end{align*}
using \eqref{eq:BJJK_eq3.3}.
For the corresponding lower bound, by Lemma~\ref{thm:sumprod1}\ref{thm:sumprod1_b},
\begin{align*}
&\sum_{\substack{ r_{i_{j-1}+1}, \ldots, r_{i_j} = \tau_N(t_{j-1})+1 
        \\ \text{all distinct} }}^{\tau_N(t_j)}  
        \prod_{i=i_{j-1}+1}^{i_j} c_N(r_i)  \\ 
&\geq \Bigg[(t_j - t_{j-1})^{i_j - i_{j-1}} - \Bigg(c_N(\tau_N(t_{j - 1}))  \\
        &\phantom{\geq \Bigg[} + \binom{i_j - i_{j-1}}{2} \sum_{s=\tau_N(t_{j-1})+1}^{\tau_N(t_j)}
        c_N(s)^2 \Bigg) ( t_j + 1 )^{i_j - i_{j-1}} \Bigg] \1{\{c_N(\tau_N(t_{j - 1})) \leq t_j - t_{j - 1}\}}.
\end{align*}
Define the events
\begin{align*}
E_N^4(j)
&= \Bigg\{ \sum_{s=\tau_N(t_{j-1})+1}^{\tau_N(t_j)} c_N(s)^2
        \leq \frac{1}{2(i_j - i_{j - 1})!} \Bigg( \frac{ t_j - t_{j-1} }{ t_j +1 }\Bigg)^{i_j - i_{j-1}} \Bigg\}, \\
F_N^4(j)
&= \Bigg\{ c_N(\tau_N(t_{j-1}))
        \leq \frac{1}{2(i_j - i_{j - 1})!} \Bigg( \frac{ t_j - t_{j-1} }{ t_j +1 }\Bigg)^{i_j - i_{j-1}} \Bigg\},    
\end{align*}
where the upper bound on the right-hand sides is strictly positive since $t_j > t_{j - 1}$, and thus satisfies the conditions of Lemmata \ref{thm:indicators_cN} and \ref{thm:indicators_c2}.
Define the event 
\begin{equation*}
    E_N^4 := \bigcap_{j=1}^k [E_N^4(j) \cap F_N^4(j) \cap \{c_N(\tau_N(t_{j-1})) \leq t_j - t_{j - 1}\}],
\end{equation*}
on which the factors of the following product are non-negative:
\begin{align*}
&\prod_{j=1}^k \frac{1}{(i_j - i_{j-1})!}
        \sum_{\substack{ r_{i_{j-1}+1}, \ldots, r_{i_j} = \tau_N(t_{j-1})+1 
        \\ \text{all distinct} }}^{\tau_N(t_j)} 
        \,\prod_{i=i_{j-1}+1}^{i_j} c_N(r_i) \\
&\geq \prod_{j=1}^k \Bigg\{ \frac{(t_j - t_{j-1})^{i_j - i_{j-1}}}{(i_j - i_{j-1})!}  - \Bigg( c_N(\tau_N(t_{j-1})) + \sum_{s=\tau_N(t_{j-1})+1}^{\tau_N(t_j)} c_N(s)^2 \Bigg)
        ( t_j +1 )^{i_j - i_{j-1}} \Bigg\} \1{E_N^4} \\
&= \sum_{\mathcal{I} \subseteq [k]} (-1)^{k-|\mathcal{I}|} 
        \Bigg( \prod_{j\in\mathcal{I}} \frac{(t_j - t_{j-1})^{i_j - i_{j-1}}}
        {(i_j - i_{j-1})!} \Bigg) \\
    &\phantom{\geq \prod_{j=1}^k \Bigg\{} \times \Bigg( \prod_{j\notin\mathcal{I}} \Bigg( c_N(\tau_N(t_{j - 1})) +
        \sum_{s=\tau_N(t_{j-1})+1}
        ^{\tau_N(t_j)} c_N(s)^2 \Bigg) ( t_j +1 )^{i_j - i_{j-1}} \Bigg) 
        \1{E_N^4} \\
&\geq \prod_{j=1}^k \frac{(t_j - t_{j-1})^{i_j - i_{j-1}}}{(i_j - i_{j-1})!} 
        \1{E_N^4} \notag\\
    &\phantom{=}- \sum_{\mathcal{I} \subset [k]} \Bigg( \prod_{j\in\mathcal{I}} 
        t^{i_j - i_{j-1}} \Bigg) 
        \Bigg( \prod_{j\notin\mathcal{I}} \Bigg( c_N(\tau_N(t_{j - 1} )) + \sum_{s=\tau_N(t_{j-1})+1}
        ^{\tau_N(t_j)} c_N(s)^2 \Bigg) ( t+1 )^{i_j - i_{j-1}} \Bigg),
\end{align*}
where the last expression separates out the $\mathcal{I} = [k]$ term, ensures all other terms have a negative sign, and bounds their magnitude from above.
Using parts \ref{item:cN_property1} and \ref{item:cN_property4} of Proposition~\ref{thm:cN_properties} to upper bound all but one of the $c_N(\tau_N(t_{j-1})) + \sum c_N(s)^2 \leq t + 2$ factors, arbitrarily setting $j^\star(\mathcal{I}) := \min\{j \notin \mathcal{I} \}$, as well as recalling that $\sum_{ j = 1 }^k (i_j - i_{j - 1}) = k$ and $|\mathcal{I}^c| \leq k$, this is further bounded by
\begin{align*}
&\prod_{j=1}^k \frac{1}{(i_j - i_{j-1})!}
        \sum_{\substack{ r_{i_{j-1}+1}, \ldots, r_{i_j} = \tau_N(t_{j-1})+1 
        \\ \text{all distinct} }}^{\tau_N(t_j)} 
        \,\prod_{i=i_{j-1}+1}^{i_j} c_N(r_i) \notag\\
&\geq \prod_{j=1}^k \frac{(t_j - t_{j-1})^{i_j - i_{j-1}}}{(i_j - i_{j-1})!} 
        \1{E_N^4} - \sum_{\mathcal{I} \subset [k]} 
        \Bigg( c_N(\tau_N(t_{j^*(I)-1})) + \sum_{s=\tau_N(t_{j^\star(\mathcal{I})-1})+1}
        ^{\tau_N(t_{j^\star(\mathcal{I})})} c_N(s)^2 \Bigg)
        ( t+2 )^{2 k} .
\end{align*}
We can now evaluate the limit:
\begin{align*}
&\lim_{N\to\infty} \E \Bigg[ \1{E_N} \sum_{\substack{r_1<\dots<r_k :\\ 
        r_i\leq \tau_N(t_i) \forall i}} \prod_{i=1}^k c_N(r_i) \Bigg] \\
&\geq \sum_{\substack{i_1\leq \dots\leq i_{k-1}\\ \in \{0,\dots,k\} 
        : i_j \geq j}}
        \prod_{j=1}^k \frac{(t_j - t_{j-1})^{i_j - i_{j-1}}}{(i_j - i_{j-1})!}   
        \lim_{N\to\infty} \Prob\left[ E_N \cap E_N^4 \right] \notag\\
    &\phantom{\geq} - \sum_{\substack{i_1\leq \dots\leq i_{k-1}\\ \in \{0,\dots,k\} 
        : i_j \geq j}}
        \sum_{\mathcal{I} \subset [k]}
        \lim_{N\to\infty} \E \Bigg[ c_N(\tau_N(t_{j^*(I)-1})) + \sum_{s=\tau_N(t_{j^\star(\mathcal{I}) -1})+1}^{\tau_N(t_{j^\star(\mathcal{I}) })}         
        c_N(s)^2 \Bigg] ( t+2 )^{2k} \notag\\
&= \sum_{\substack{i_1\leq \dots\leq i_{k-1}\\ \in \{0,\dots,k\} 
        : i_j \geq j}}
        \prod_{j=1}^k \frac{(t_j - t_{j-1})^{i_j - i_{j-1}}}{(i_j - i_{j-1})!},
\end{align*}
where we used \eqref{eq:BJJK_eq3.3} and \eqref{eq:BJJK_eq3.5} to conclude that the sum in the expectation vanishes, and Lemmata \ref{thm:lim_AandB}, \ref{thm:indicators_cN}, and \ref{thm:indicators_c2} to obtain that $\lim_{N\to\infty}\Prob[E_N \cap E_N^4] =1$.
The upper and lower bounds coincide, so the proof is complete.
\end{proof}

\subsection{Indicators}
Many of the preceding results make use of indicator functions in order to control the sign of certain terms.
The probabilities of the corresponding events were claimed to converge to $1$ as $N\to\infty$, so that the indicators do not affect the limit.
These claims are proved in here.
Firstly, Lemma~\ref{thm:lim_AandB} was proved in \cite{brown2021thesis}, and shows that suffices to prove the limits separately for each factor in a product of indicators of two or more events.

\begin{lemma}{\cite[Lemma 4.11]{brown2021thesis}}\label{thm:lim_AandB}
Let $(A_N) , (B_N)$ be sequences of events. 
If $\lim_{N\to\infty}\Prob[A_N] = \lim_{N\to\infty}\Prob[B_N]=1$ then $\lim_{N\to\infty} \Prob [A_N \cap B_N] =1$.
\end{lemma}

The remainder of this section is split into four lemmata, each showing that the probabilities of certain events converge to $1$ as $N\to\infty$.
The first three are variants of \citet[Supplement, Lemma 4]{koskela2022erratum}, with analogous proofs.
For completeness, self-contained proofs of Lemmata \ref{thm:indicators_cN} -- \ref{thm:indicators_DN} can be found in \citet[Lemmata 4.12--4.14]{brown2021thesis}.

\begin{lemma}\label{thm:indicators_cN}
Assume \eqref{eq:BJJK_eq3.5} holds. Fix $t>0$.
Let $K_N>0$ be a sequence independent of $r$ and bounded away from 0.
Define $E_N(r) := \{ c_N(r) < K_N \}$ and $E_N := \bigcap_{r=1}^{\tau_N(t)} E_N(r)$.
Then $\lim_{N\to\infty} \Prob[E_N]=1$.
\end{lemma}

\begin{lemma}\label{thm:indicators_tau}
Assume \eqref{eq:BJJK_eq3.5} holds.
Fix $t > 0$.
For any $l \in \mathbb{N}$, $\Prob[ \tau_N(t) \geq l ] \to 1$ as $N \to \infty$.
\end{lemma}

\begin{lemma}\label{thm:indicators_DN}
Assume \eqref{eq:BJJK_eq3.4} holds.
Fix $t>0$.
Let $K$ be a constant not depending on $N$ or $r$.
Then
\begin{equation*}
\lim_{N\to\infty} \Prob \Bigg[ \bigcap_{r=1}^{\tau_N(t)} 
        \left\{ c_N(r) \geq K D_N(r) \right\} \Bigg] 
= 1 . 
\end{equation*}
\end{lemma}

\begin{lemma}\label{thm:indicators_c2}
Assume \eqref{eq:BJJK_eq3.5} holds.
Fix $k\in\mathbb{N}$, a sequence of times $0 = t_0 \leq t_1 \leq \cdots \leq t_k \leq t$, and let $K_{ 1 : k }$ be strictly positive constants.
Define the event
\begin{equation*}
E_N := \bigcap_{j=1}^k \Bigg\{ \sum_{s=\tau_N(t_{j-1})+1}^{\tau_N(t_j)}
        c_N(s)^2 \leq K_j \Bigg\} .
\end{equation*}
Then $\lim_{N\to\infty} \Prob[E_N] =1$.
\end{lemma}

\begin{proof}
\begin{align*}
\Prob[E_N]
&= 1- \Prob[E_N^c]
= 1- \Prob\Bigg[ \bigcup_{j=1}^k \Bigg\{ \sum_{s=\tau_N(t_{j-1})+1}
        ^{\tau_N(t_j)} c_N(s)^2 > K_j \Bigg\} \Bigg] \\
&\geq 1- \sum_{j=1}^k \Prob\Bigg[ \sum_{s=\tau_N(t_{j-1})+1}
        ^{\tau_N(t_j)} c_N(s)^2 \geq K_j \Bigg] .
\end{align*}
Applying Markov's inequality and using \eqref{eq:BJJK_eq3.5} gives
\begin{equation*}
\Prob[E_N]
\geq 1- \sum_{j=1}^k K_j^{-1} \E\Bigg[ \sum_{s=\tau_N(t_{j-1})+1}
        ^{\tau_N(t_j)} c_N(s)^2 \Bigg]
\underset{N\to\infty}{\longrightarrow} 1- \sum_{j=1}^k K_j^{-1} \times 0 
= 1,
\end{equation*}
as required.
\end{proof}

\subsection{Fubini \& dominated convergence conditions}

There are a few instances where Fubini's Theorem and the Dominated Convergence Theorem are used to pass a limit and an expectation through an infinite sum.
This result, whose proof is a simple adaptation of the argument on \citet[Supplement, p.\ 33--34]{koskela2022erratum}, verifies the conditions of those theorems.
It is used in \eqref{eq:022}, in \eqref{eq:029}, in Lemma~\ref{thm:inductionUB} at \eqref{eq:040}, and in Lemma~\ref{thm:inductionLB} at \eqref{eq:056}.
\begin{lemma}\label{thm:DCT_Fubini}
For any fixed $t>0$, for $N$ sufficiently large,
\begin{equation*}
\E \Bigg[ \sum_{l=0}^\infty \Big| (-\alpha_n)^l \ON \frac{1}{l!} t^l 
        \sum_{\substack{r_1 < \cdots < r_k : \\ r_i \leq \tau_N(t_i) \forall i}}
        \prod_{i=1}^k c_N(r_i) \Big| \Bigg]
< \infty .
\end{equation*}
\end{lemma}

\section*{Acknowledgements}

This work was supported by the Engineering and Physical Sciences Research Council, the Medical Research Council, the Alan Turing Institute, and the Alan Turing Institute--Lloyd’s Register Foundation Programme on Data-centric Engineering, under grant numbers EP/L016710/1, EP/N510129/1, EP/R034710/1, EP/R044732/1, EP/T004134/1 and EP/V049208/1.
Data sharing not applicable to this article as no datasets were generated or analysed during the current study.

\bibliographystyle{elsarticle-harv} 
\bibliography{smc.bib}
\end{document}